\documentclass[a4paper,10pt]{article}
\usepackage[utf8]{inputenc}

\usepackage{amssymb}
\usepackage{amsthm}
\usepackage{latexsym}
\usepackage{amsfonts}
\usepackage{amsmath}
\usepackage{cancel}
\usepackage{mathrsfs}
\usepackage{graphicx}
\usepackage{color}
\usepackage{amsmath,amscd}
\usepackage{verbatim}
\usepackage{mathdots}
\usepackage{typearea}
\usepackage{geometry}
\usepackage{url}
\usepackage{afterpage}
\usepackage{enumerate}
\usepackage{hyperref}

\input diagrams.sty

\diagramstyle[labelstyle=\scriptstyle]

\newcounter{BiB}
\newtheorem{thm}{Theorem}[subsection]
\newtheorem*{thmNoNumber}{Theorem}
\newtheorem{lemma}[thm]{Lemma}
\newtheorem{prop}[thm]{Proposition}

\newtheorem{cor}[thm]{Corollary}

\DeclareMathAlphabet{\mathpzc}{OT1}{pzc}{m}{it}

\renewcommand{\\}{\newline}
\newcommand{\A}{\mathbb{A}}

\newcommand{\Borelbf}{\text{\rm\bf B}}
\newcommand{\B}{\text{\rm\bf B}}

\newcommand{\CDrin}{\mathcal{Y}^{(n+1)}}

\newcommand{\coker}{\mbox{\rm coker}}

\newcommand{\Drin}{\mathcal{X}^{(n+1)}}
\newcommand{\dlim}{\mathop{\varinjlim}\limits}
\newcommand{\ECal}{\mathcal{E}}

\newcommand{\Ext}{\mbox{\rm Ext}}
\newcommand{\F}{\mathbb{F}}
\newcommand{\FCal}{\mathcal{F}}

\newcommand{\G}{\text{{\rm\bf G}}}
\newcommand{\gr}{\mbox{{\rm gr}}}

\newcommand{\GCal}{\mathcal{G}}

\newcommand{\GL}{\text{\rm GL}}
\newcommand{\SL}{\text{\rm SL}}
\newcommand{\HCal}{\mathcal{H}}
\newcommand{\HH}{\text{\rm H}}
\newcommand{\Hyp}{\mathbb{H}}
\newcommand{\Hom}{\mbox{\rm Hom}}
\newcommand{\id}{\text{\rm id}}

\newcommand{\Ind}{\mbox{\rm Ind}}

\newcommand{\N}{\mathbb{N}}

\newcommand{\Om}{\mathcal{O}}

\newcommand{\Parbf}{\text{\rm \bfseries P}}

\newcommand{\PR}{\mathbb{P}}

\newcommand{\Proj}{\mathop{\rm\bf Proj }}
\newcommand{\Q}{\mathbb{Q}}

\newcommand{\R}{\mathbb{R}}
\newcommand{\sad}{\text{\rm ad }}
\newcommand{\spec}{\text{\rm sp}}
\newcommand{\Spec}{\mathop{\rm\bf Spec}}

\newcommand{\Tbf}{\text{\rm\bf T}}

\newcommand{\Tot}{\mbox{\rm Tot}}

\newcommand{\VCal}{\mathcal{V}}

\newcommand{\Z}{\mathbb{Z}}

\title{Rigid Cohomology of Drinfeld's Upper Half Space over a Finite Field}
\author{Mark Kuschkowitz\\Email: mark.kuschkowitz@googlemail.com}

\begin{document}

\maketitle

\begin{abstract}
In this paper the rigid cohomology of Drinfeld's upper half space over a finite field is computed in two ways. The first method proceeds by computation of the rigid cohomology of the complement of Drinfeld's upper half space in the ambient projective space and then use of the associated long exact sequence for rigid cohomology with proper supports. The second method proceeds by direct computation of rigid cohomology as a direct limit of de Rham cohomologies of a family of strict open neighborhoods of the tube of Drinfeld's upper half space in the ambient rigid-analytic projective space. The resulting cohomology formula has been known since 2007, when Gro{\ss}e-Kl\"onne proved that it is the same as the one obtained from $\ell$-adic cohomology.
\end{abstract}

\allowdisplaybreaks

\section{Introduction}

Fix a prime number $p$ and a finite field extension $k$ of the field $\mathbb{F}_p$ with $p$ elements. The $n$-dimensional Drinfeld upper half space over $k$ is defined as
$$\Drin=\Drin_k=\PR_k^n\setminus\bigcup_{H\subsetneq k^{n+1}}\PR(H),$$
i.e.\ it arises by removing all $k$-rational hyperplanes from the projective $n$-space $\PR_k^n$ over $k.$ This affine space was first defined and studied by Drinfeld \cite{DrinfeldEM,DrinfeldCovs} and it arose in the study of $\ell$-adic representations of finite groups of Lie type (for a prime $\ell\neq p$). In particular, the finite group $\GL_{n+1}(k)$ acts on $\Drin$ (since it permutes the hyperplanes) and Drinfeld showed that the $\ell$-adic cohomology of $\mathcal{X}^{(2)}$ realizes -- up to a sign -- all cuspidal representations with $\ell$-adic coefficients of the group $\SL_2(k).$ Deligne and Lusztig \cite{DeLu} took up this idea and defined varieties -- now called Deligne-Lusztig varieties -- associated with a reductive group $G$ over $k$ whose $\ell$-adic cohomology realizes all irreducible $\ell$-adic representations of the finite group $G(k)$ of $k$-rational points of $G.$ This incorporated Drinfeld's work, as $\Drin$ is in particular isomorphic to a Deligne-Lusztig variety for the algebraic $k$-group $\GL_{n+1,k}$ (associated with the standard Coxeter element of its Weyl group). 

The $\ell$-adic cohomology of $\Drin\times_k \overline k$ was computed by Orlik \cite{OrlikDiss} as a representation of $\GL_{n+1}(k)$ and as a Galois representation. In particular, Orlik's complex, which will be used extensively in this paper, appeared for the first time in \cite{OrlikDiss}. In fact, this particular result was known before since Kottwitz and Rapoport had computed the Euler-Poincar\'e characteristic, cf.\ \cite{RapoportICM,RapoportSantaCruz}, and then the application of purity results yields the cohomology already.
 
Gro{\ss}e-Kl\"onne studied in \cite{GrKl} the rigid cohomology of Deligne-Lusztig varieties and therefore in particular of $\Drin$. Among other results, he found that after identifying the respective coefficient fields, this cohomology coincides with the $\ell$-adic cohomology, seen essentially as a module over $\GL_{n+1}(k).$

It should also be mentioned that the de Rham cohomology in the analogous situation of a $p$-adic local base field, i.e. if one considers Drinfeld's upper half space over a finite extension field of the field $\Q_p$ of $p$-adic numbers, had already been computed by Schneider and Stuhler in \cite{SchSt}.\\

Fix a finite field extension $K$ of the field $\Q_p$ of $p$-adic numbers with valuation ring $\VCal$ and residue field $k.$ Denote by $\G$ the algebraic $\VCal$-group scheme $\GL_{n+1,\VCal}$ and by $\G_k$ its base change to $k.$ The objective of this paper is the computation of the rigid cohomology $\HH_{\rm rig}^{\ast}(\Drin\slash K)$ (resp.\ -- by Poincar\'e duality -- the rigid cohomology $\HH_{\rm rig,c}^{\ast}(\Drin\slash K)$ ``with compact supports'') of $\Drin$ with coefficients in $K$ as a module over the finite group $G=\G(k)=\G_k(k).$ 

For $i\in\{0,\ldots,n\},$ let $\Parbf_{(i+1,1^{n-i})}$ be the lower parabolic subgroup of $\G$ associated with the decomposition $n+1=(i+1)+(n-i)\cdot 1.$ Denote by $v_{\Parbf_{(i+1,1^{n-i})}(k)}^{G}(K)$ the generalized Steinberg representation of $G$ with coefficients in $K$ and by $v_{\Parbf_{(i+1,1^{n-i})}(k)}^{G}(K)'$ its $K$-dual\footnote{See Section \ref{Prelims} for an account of the notation used in this section.}. The result is the following theorem. 
\begin{thmNoNumber}(see Theorem \ref{RigDrinFirstComp} resp.\ Theorem \ref{RigDrinSecComp})
 The rigid cohomology of $\Drin$ with coefficients in $K$ is
 $$\HH_{\rm rig}^{\ast}(\Drin\slash K)=\bigoplus_{i=0}^n{v_{\Parbf_{(i+1,1^{n-i})}(k)}^{G}}(K)'(i-2n)[-n+i].$$
\end{thmNoNumber}

The notation $[-i]$ means that the respective module lives in degree $i$ and $(-)$ determines its Tate twist. 

As stated above, this theorem is due to Gro{\ss}e-Kl\"onne \cite{GrKl}. Here, this cohomology is computed once indirectly by considering the complement of $\Drin$ in $\PR_k^n$ and once directly from the associated de Rham complex. Both times, modified versions of Orlik's complex are used. For technical reasons, many computations are carried out in the category of Huber's adic spaces.\\

The next two subsections contain an outline of both computations.

\subsubsection{Rigid Cohomology of \texorpdfstring{$\Drin$}{X} Computed Indirectly}

Recall that for a projective $k$-variety $Y\subset\PR_k^n,$ its rigid cohomology is simply the de Rham cohomology $\HH_{\rm dR}^\ast(]Y[,K)$ with values in $K$ of its rigid analytic tube (of radius $1$) $]Y[\subset \PR_K^{n,\rm rig}.$ Therefore, in this first part, the de Rham cohomology of the rigid analytic tube $]\CDrin[$ of $\CDrin=\PR_k^n\setminus \Drin$ is computed. The rigid cohomology of $\Drin$ is then known by applying the long exact cohomology sequence for the pair of inclusions $\Drin\stackrel{\text{open}}{\subset}\PR_k^n\stackrel{\text{closed}}{\supset}\CDrin.$ The key ingredient for the computation of the rigid cohomology $\HH_{\rm rig}^\ast(\CDrin\slash K)=\HH_{\rm dR}^\ast(]\CDrin[,K)$ is a modified version of Orlik's complex for $]\CDrin[$ with values in the de Rham complex $\Omega_{]\CDrin[\slash K}^{\bullet}$ on $]\CDrin[.$ In this situation, Orlik's complex is indexed by the tubes associated with certain closed subvarieties of $\CDrin$ and this complex is shown to be acyclic, see Corollary $\ref{FundaKompRevArrCorr}.$ Thus it induces a spectral sequence converging to $\HH_{\rm rig}^{\ast}(\CDrin\slash K)$ with computable entries on the $E_1$-page. This spectral sequence degenerates on its $E_2$-page (Lemma \ref{ETO}) and the evaluation of the associated grading on $\HH_{\rm rig}^{\ast}(\CDrin\slash K)$ then yields this cohomology (Proposition \ref{RigCohCDrinComp}).

\subsubsection{Rigid Cohomology of \texorpdfstring{$\Drin$}{X} Computed from the Associated De Rham Complex}

For a quasi-projective $k$-variety $X\subset \PR_k^n,$ its rigid cohomology $\HH_{\rm rig}^{\ast}(X\slash K)$ is defined as the hypercohomology on $\PR_K^{n,\rm rig}$ with values in the direct limit of the de Rham complexes of a system of strict open neighborhoods of $]X[$ (the rigid-analytic tube of $X$ in $\PR_K^{n,\rm rig}$). In the first step towards computing $\HH_{\rm rig}^{\ast}(\Drin\slash K),$ a cofinal system $(U^m)_{m\in\N}$ of affinoid strict open neighborhoods of $]\Drin[$ in $\PR_K^{n,\rm rig}$ is constructed, where
$$U^m = \PR_K^{n,\rm rig}\setminus\bigcup_{H\subsetneq k^{n+1}}]\PR(H)[_{\lambda_m}$$
for $\lambda_m=\#(k\slash \F_p)^{-1\slash m+1}$ and $]Z[_{\lambda_m}$ is the tube of $Z$ with radius $\lambda_m$ in $\PR_K^{n,\rm rig}$ (see Lemma \ref{LostLemma}). It is then shown that in order to use the associated de Rham complex for the determination of the above hypercohomology, it is enough to compute the spaces of sections $\HH^0(U^m,\Omega_{\PR_K^{n,\rm rig}\slash K}^i)$ of the sheaves of differential $i$-forms $\Omega_{\PR_K^{n,\rm rig}\slash K}^i$ on each $U^m$ and then take their direct limit (Lemma \ref{DERHAMWORKS}). The spaces $U^m$ are constructed in such a way that one can again make use of an adapted version of Orlik's complex to determine all spaces $\HH^0(U^m,\Omega_{\PR_K^{n,\rm rig}\slash K}^i).$ To be precise, Orlik's complex in this situation (which is again acyclic, see Proposition \ref{ACDC}) yields a spectral sequence which computes the local cohomology $\HH_{Y^m}^1(\PR_K^{n,\rm rig},\Omega_{\PR_K^{n,\rm rig}\slash K}^i)$ of $\PR_K^{n,\rm rig}$ with supports in the complement $Y^m=\PR_K^{n,\rm rig}\setminus U^m$ of $U^m$ in $\PR_K^{n,\rm rig}$ and values in $\Omega_{\PR_K^{n,\rm rig}\slash K}^i.$ The key to the evaluation of this spectral sequence is again the fact that it has computable entries in its $E_1$-page, see Lemma \ref{EONELEMMA}. The remainder of the evaluation of this spectral sequence then proceeds analogous to the respective section of Orlik's paper \cite{OrlikEVB}. In particular, it yields a finite filtration of $\G(\VCal)$-submodules of $\HH_{Y^m}^1(\PR_K^{n,\rm rig},\Omega_{\PR_K^{n,\rm rig}\slash K}^i),$ see Theorem \ref{DescMod}.

The rest of the computation then closely follows Orlik's paper \cite{OrlikDeRham}: The functorialities of the filtrations obtained in the last theorem then imply that the de Rham complex 
$$0\rightarrow \HH^0(U^m,\Om_{\PR_K^{n,\rm rig}})\rightarrow \HH^0(U^m,\Omega_{\PR_K^{n,\rm rig}}^1)\rightarrow \ldots\rightarrow \HH^0(U^m,\Omega_{\PR_K^{n,\rm rig}}^n)\rightarrow 0$$
is actually a filtered complex. The associated spectral sequence degenerates on its $E_1$-page with computable entries. Taking direct limits, one then obtains the desired rigid cohomology. The key result is Lemma \ref{KEYLEMMA} which asserts that for each $j\in\{0,\ldots,n-1\},$ the complex  
\begin{eqnarray*}
	0&\rightarrow& \dlim_{m\in\N}\tilde\HH_{]\PR_k^{j}[_{ 	\lambda_m}}^{n-j}(\PR_K^{n,\rm rig},\Omega_{\PR_K^{n,\rm rig}\slash K}^0)\rightarrow \dlim_{m\in\N}\tilde\HH_{]\PR_k^{j}[_{\lambda_m}}^{n-j}(\PR_K^{n,\rm rig},\Omega_{\PR_K^{n,\rm rig}\slash K}^1)\rightarrow\ldots\\
	&\rightarrow&\dlim_{m\in\N}\tilde\HH_{]\PR_k^{j}[_{\lambda_m}}^{n-j}(\PR_K^{n,\rm rig},\Omega_{\PR_K^{n,\rm rig}\slash K}^n)\rightarrow 0
\end{eqnarray*}
consisting of direct limits of reduced local cohomologies with support in $]\PR_k^{j}[_{\lambda_m}$ is acyclic. 

\subsection{Structure of this Paper}

For the reader's convenience, a short overview of the organization of the content of this paper is given in the sequel:

In Section \ref{Prelims} the general notation used in the course of this paper is fixed. Furthermore, there are some short recollections on (generalized) Steinberg representations, on Drinfeld's upper half space and on rigid geometry and Huber's adic spaces. 

In Section \ref{SecThreeOne} the construction of rigid cohomology (of a quasi-projective $k$-variety) is reviewed and some of its properties are recalled. 

In Sections \ref{SecThreeThree} and \ref{SecThreeFour} the computations of $\HH_{\rm rig}^*(\Drin\slash K)$ are carried out. Each section has the following structure: Orlik's complex is adapted for certain classes of tubes of rigid varieties (Subsections \ref{SecThreeTwo} and \ref{OneMoreTime}), a spectral sequence is set up (Subsections \ref{ThreeThreeOne} and \ref{ThreeFourOne}), computed (Subsections \ref{ThreeThreeTwo} and \ref{ThreeFourTwo}) and the result is used to finally compute the rigid cohomology modules of $\Drin$ (Subsections \ref{ThreeThreeThree} and \ref{ThreeFourThree}).

\subsection{Acknowledgements}

The content of this paper is a part of my doctoral dissertation at Bergische Universit\"at Wuppertal and I wish to thank all my former colleagues in the working group ``Arbeitsgemeinschaft Algebra und Zahlentheorie'' for creating a wonderful working atmosphere. In particular, I would like to mention and sincerely thank my advisor Sascha Orlik who introduced me to the subjects treated here and who offered invaluable help and encouragement. Furthermore, Roland Huber and Markus Reineke always patiently answered my questions and my friends and colleagues Martin Bender and Hans Franzen were always wonderful discussion partners and sounding boards; I wish to thank them all for their help and support.

\section{Notation and Preliminaries}\label{Prelims}

As in the introduction, fix a prime number $p,$ denote by $\F_p$ the field with $p$ elements and by $k$ a fixed finite field extension of $\F_p.$ Fix a finite field extension $K\slash \Q_p$ with residue field $k,$ denote by $\VCal$ its valuation ring and let $\pi\in\VCal$ be a uniformizing element. Choose a norm $\lvert\text{ }\rvert:K\rightarrow\R_{\geq 0},$ normalized such that $\lvert\pi\rvert =\#(k\slash \F_p)^{-1}.$ Fix a completion $\widehat{\overline K}$ of an algebraic closure $\overline K$ of $K$ and also denote by $\lvert\text{ }\rvert:\widehat{\overline K}\rightarrow\R_{\geq 0}$ the uniquely determined extension of $\lvert\text{ }\rvert$ to $\widehat{\overline K}.$

\subsection{(Algebraic) Groups and Representations}\label{AlgGrpReps}

Fix $n\in\N$ and let $\G$ be the algebraic group scheme $\GL_{n+1,\VCal}$ over $\VCal$ with lower Borel subgroup $\B$ (i.e.\ for a $\VCal$-algebra $R,$ the group $\B(R)$ is the subgroup of lower triangular matrices in $\G(R)$) and diagonal torus $\Tbf\subset\B.$ Write
$$X(\Tbf)=\Hom(\Tbf,\G_m)$$ 
for the $\Z$-module of algebraic characters of $\Tbf.$ For $i\in\{0,\ldots,n\},$ denote by $\epsilon_i\in X(\Tbf)$ the character which sends an element $(t_0,\ldots,t_n)\in\Tbf(R),$ $R$ a $\VCal$-algebra, to $t_i.$ For $0\leq i\neq j\leq n,$ let 
$$\alpha_{i,j}=\epsilon_i-\epsilon_j.$$ 
Then 
$$\Phi=\{\alpha_{i,j}\mid 0\leq i\neq j\leq n\}$$ 
is the set of roots of $\G$ (with respect to $\Tbf$) and
$$\Delta=\{\alpha_{0}=\alpha_{1,0},\ldots,\alpha_{n-1}=\alpha_{n,n-1}\}$$ 
is the set of simple roots (with respect to $\Borelbf\supset\Tbf$). For a proper subset $I\subsetneq \Delta,$ denote by $\Parbf_I$ the associated standard-parabolic subgroup (with respect to $\Borelbf$). Whenever it is more convenient, the alternative description of standard-parabolic subgroups in terms of decompositions of $n+1$ will be used: Let $(i_0,\ldots,i_r)$ be a decomposition of $n+1,$ i.e.\ $n+1=i_0+\ldots+i_r$ with $i_0,\ldots,i_r\in\N.$ Then $(i_0,\ldots,i_r)$ defines a standard-parabolic subgroup $\Parbf_{(i_0,\ldots,i_r)}=\Parbf_I$ with 
\begin{eqnarray*}
I&=&\{\alpha_{0},\ldots,\alpha_{i_0-2}\}\cup \{\alpha_{i_0},\ldots,\alpha_{i_0+i_1-2}\}\cup\ldots\cup \{\alpha_{i_0+\ldots+i_{r-1}},\ldots,\alpha_{i_0+\ldots+i_r-2}\}\nonumber                                                                                                                                                                                                                                                                                                                                                                                                                                                                                                                                                                                                                                                                                                                                                                                                                                                                                                                                                                                                                                                                                                                                                                                                                                                                                                                                                                                                                           \end{eqnarray*}
(where all undefined sets are to be understood as being empty).

For a $\VCal$-algebra $R,$ write 
$$\G_R=\G\times_{\Spec\VCal}\Spec R$$ 
and similarly for the subgroups mentioned above. In particular, in the case that $R=k,$ set
$$G=\G_k(k),B=\Borelbf_k(k),P_I=\Parbf_{I,k}(k),T=\Tbf_k(k)$$
and similarly in the case where $I$ is replaced by its associated decomposition of $n+1$ in the above sense.

\subsection{(Generalized) Steinberg Representations}\label{GenSteinRep}

Suppose that $H=P_I$ is the group of $k$-points of a standard-parabolic subgroup $\Parbf_I\subset\G.$ Equip $K$ with the trivial action of $G.$ The generalized Steinberg representation of $G$ with respect to $K$ and $H$ is defined as
$$v_H^G(K)=\Ind_H^G(K)\slash\sum_{H\subsetneq H'\subset G}\Ind_{H'}^G(K).$$
If $H=B,$ then $v_H^G(K)$ is called Steinberg representation. This representation is irreducible and self-dual, cf. \cite{HumSt}. The generalized Steinberg representation is irreducible if and only if $I$ is either empty or of the shape 
$I=\{\alpha_0,\alpha_1,\ldots,\alpha_i\},$ $i\in\{0,\ldots,n-1\},$ cf.\ \cite[Prop. 2.5]{OrlikDeLu}.

\subsection{Drinfeld's Upper Half Space over a Finite Field}\label{DrinUpHaSp}

Let 
$$S=k[T_0,\ldots,T_n]$$ 
be the polynomial ring in $n+1$ variables over $k$ with its usual grading $S=\bigoplus_{i\in\N_0}S_i$ and write 
$$\PR_k^n=\Proj S.$$ 
For $j\in \{0,\ldots,n-1\},$ denote by $\PR_k^j$ the closed subvariety $V_+(T_{j+1},\ldots,T_n)\subset\PR_k^n.$ The $n$-dimensional Drinfeld upper half space over $k$ is the affine open subvariety
$$\Drin=\PR_k^n\setminus\bigcup_{f\in S_1}V_+(f)$$
of $\PR_k^n$ which arises by removing all $k$-rational hyperplanes from $\PR_k^n.$ Denote by $\CDrin$ its closed complement in $\PR_k^n,$ i.e.\
$$\CDrin=\bigcup_{f\in S_1}V_+(f).$$

Consider the action of $\G_k$ on $\PR_k^n$ which on closed points is given by $$(g,[x_0:\ldots:x_n])\mapsto [x_0:\ldots:x_n]g^{-1}.$$ 
This action restricts to an action of the finite group $G$ on $\Drin$ resp.\ on $\CDrin$ since it permutes the hyperplanes $V_+(f)$ (with $f\in S_1$).

\subsection{Rigid Analytic Varieties and Adic Spaces}\label{RigVarAdSp}

Throughout this paper, there will be made use of the concepts of rigid-analytic spaces over $K$ (cf.\ e.g.\ \cite{BGR,FrPu}) and of adic spaces over $K$ in the sense of Huber (cf.\ e.g.\ \cite{HuberMZ,HuberBook}). The following notation and facts will be used: For a $K$-variety $X,$ denote by $X^{\rm rig}$ its associated rigid-analytic $K$-variety (as in \cite[9.3.4]{BGR}) and by $X^{\sad}$ its associated adic space (as in \cite[Par.\ 4]{HuberMZ}). If $X$ is a rigid-analytic $K$-variety, its associated adic space will also be denoted by $X^{\sad},$ cf.\ loc.\ cit. To be a little more precise, recall that there are the following functors: 
\begin{itemize}
 \item $(-)^{\rm rig},$ from the category of $K$-varieties to the category of rigid analytic $K$-varieties,
 \item $(-)^{\rm ad},$ from the category of rigid analytic $K$-varieties to the category of adic spaces over $K,$
 \item $(-)^{\rm ad},$ from the category of formal schemes over $K$ to the category of adic spaces over $K,$
\end{itemize}
such that for a $K$-variety $X$, there is an isomorphism $(X^{\rm rig})^{\sad}\cong X^{\sad}$ (where on the right-hand side, $X$ is considered as a formal scheme, i.e. as the formal completion along itself). Furthermore, each of these functors induces an equivalence of the respective topoi of sheaves, which are, respectively, the topos of sheaves on the Zariski topology of a $K$-variety, the topos of sheaves on the Grothendieck site associated with a rigid-analytic $K$-variety, and the topos of sheaves of an adic space over $K$. So in particular, for a rigid-analytic $K$-variety $X$ and a sheaf $\FCal$ on its Grothendieck topology, there is an isomorphism 
$$\HH^i(X,\FCal)\cong\HH^i(X^{\sad},\FCal^{\sad})$$ 
of cohomology groups, where $\FCal^{\sad}$ denotes the sheaf on $X^{\sad}$ induced by $\FCal$ via the above equivalence. For all facts mentioned in this paragraph, cf.\ \cite[Par. 4]{HuberMZ} resp.\ \cite[1.1.11-12]{HuberBook}. 

Concerning this paper, the main advantage of working in the category of adic spaces at times instead of the category of rigid analytic spaces is that the former are in fact topological spaces. In particular, this means that sheaves on adic spaces are well-behaved with respect to localization in points.

\section{Construction of Rigid Cohomology and some Properties}\label{SecThreeOne}

For this section, references are for example \cite{BerGeoRig,LeStumRigCoh}. 

\subsection{Berthelot's Definition of Rigid Cohomology (with and without Supports)}\label{BerDefRiCoh}

Let $\widehat{\PR_{\VCal}^n}$ be the formal completion of $\PR_{\VCal}^n$ along its special fiber $\PR_k^n.$ There is then a closed embedding $\PR_k^n\hookrightarrow \widehat{\PR_{\VCal}^n},$ which is a homeomorphism of the underlying topological spaces. Furthermore, there is a specialization map 
$$\spec^{\rm rig}:\PR_K^{n,\rm rig}\rightarrow \PR_k^n,$$ 
which on points is given by
$$[x_0:\ldots:x_n]\mapsto[\bar x_0:\ldots:\bar x_n]\in\PR_k^n$$ for $[x_0:\ldots:x_n]$ unimodular, i.e.\ $\lvert x_i\rvert \leq 1$ for all $i\in\{0,\ldots,n\}$ and $\lvert x_j\rvert=1$ for at least one $j\in\{0,\ldots,n\}.$ Here and in the sequel, $\bar x$ denotes the element of $\overline{k}$ defined by an element $x$ in the ring of integers of $\widehat{\overline K}$ by reduction modulo its maximal ideal. 

Let $X\subset\PR_k^n$ be a locally closed (quasi-projective) smooth $k$-subvariety. By definition, there is a closed subvariety $Y\subset \PR_k^n$ together with an open embedding $X\subset Y.$ Set 
$$]X[=(\spec^{\rm rig})^{-1}(X).$$
This is a rigid-analytic subvariety of $\PR_K^{n,\rm rig},$ called the tube of $X$ (of radius $1$). For the purposes of this paper, it will suffice to assume that either $Y=\PR_k^n$ or $X=Y,$ i.e.\ $X$ is either open or closed in $\PR_k^n.$  
 
Suppose first that $X$ is open in $\PR_k^n=Y$ and denote by $Z=\PR_k^n\setminus X$ its closed complement. A strict open neighborhood of $]X[$ in $\PR_K^{n,\rm rig}$ is an admissible open subset $V\subset \PR_K^{n,\rm rig}$ such that $]X[\subset V$ and such that $(V,]Z[)$ is an admissible covering of $\PR_K^{n,\rm rig}.$ The category of coefficients of rigid cohomology is the category of overconvergent $F$-isocrystals on $X\slash K$ whose objects can be briefly described as follows: An overconvergent $F$-isocrystal on $X\slash K$ is a coherent $\Om_{V}$-module $\FCal$ on some strict open neighborhood $V$ of $]X[$ in $\PR_K^{n,\rm rig}$ together with an integrable connection $\FCal\rightarrow\FCal\otimes_{\Om_V}\Omega_V^1$ such that certain overconvergence conditions are fulfilled and such that locally, there is an isomorphism $F^*\FCal\xrightarrow{\sim}\FCal,$ where $F$ denotes a local lift of the absolute ($q$-power) Frobenius
on $\PR_k^n$ to $\PR_K^{n,\rm rig},$ cf.\ e.g.\ \cite[§ 4]{BerGeoRig} or \cite[§ 7-8]{LeStumRigCoh}. Then the rigid cohomology of $X$ with values in $\FCal$ is defined as the hypercohomology
$$\HH_{\rm rig}^*(X\slash K,\FCal)=\Hyp^*\left(\PR_K^{n,\rm rig},\dlim_{V'}\left({j_{V'}}_*\right)\left(j_{V',V}^*\right)(\FCal\otimes_{\Om_{V}}\Omega_{V}^{\bullet})\right).$$
Here, the limit is taken over all strict open neighborhoods $V'$ of $]X[$ in $\PR_K^{n,\rm rig}$ which are contained in $V$ and 
\begin{eqnarray*}
j_{V',V} &:& V'\hookrightarrow V, \\
j_{V'} &:& V'\hookrightarrow \PR_K^{n,\rm rig}
\end{eqnarray*}
are the respective embeddings of admissible open subsets given by inclusion. In the case of the trivial isocrystal, i.e.\ when $\FCal$ is just the structure sheaf, write $\HH_{\rm rig}^{\ast}(X\slash K)$ for the resulting rigid cohomology.

If $X$ is closed in $\PR_k^n,$ then an overconvergent $F$-isocrystal on $X\slash K$ is an $\Om_{]X[}$-module with the above additional properties and the definition of rigid cohomology of $X$ with values in $\FCal$ simplifies to
$$\HH_{\rm rig}^*(X\slash K,\FCal)=\Hyp^*(]X[,\FCal\otimes_{\Om_{]X[}}\Omega_{]X[}^{\bullet}).$$
In particular, in the case of the trivial isocrystal, its rigid cohomology on $X$ is just the usual de Rham cohomology of $]X[.$

There is a notion of rigid cohomology ``with compact supports'', written $\HH_{\rm rig,c}^{\ast}(X\slash K,\FCal).$ Here, $\FCal$ (as above) is replaced by $$\underline\Gamma(\FCal)=\ker(\FCal\rightarrow \iota_{*}\iota^*\FCal),$$ 
where $\iota:V\cap]Z[\hookrightarrow V$ is the inclusion, and then, by definition,
$$\HH_{\rm rig,c}^{\ast}(X\slash K,\FCal)=\Hyp^{\ast}(V, {\rm \bf R}\underline\Gamma(\FCal\otimes_{\Om_{V}} \Omega_{V}^{\bullet})).$$
Again, in the case that $\FCal$ is trivial, the notation reduces to $\HH_{\rm rig, c}^{\ast}(X\slash K).$

\subsection{Some Properties of Rigid Cohomology}\label{RigProps}

First of all, the above definitions of rigid cohomology and of rigid cohomology with compact supports are essentially (i.e.\ up to canonical isomorphism) independent of all choices made and make $\HH_{\rm rig}^{\ast}$ into a Weil cohomology theory. In particular, the following hold:
\begin{itemize}
 \item For each $i\in\N_0$, the $K$-spaces $\HH_{\rm rig}^{i}(X\slash K,\FCal)$ and $\HH_{\rm rig,c}^{i}(X\slash K,\FCal)$ are finite-dimensional.
 \item If $X$ is of dimension $n,$ then $\HH_{\rm rig,c}^{i}(X\slash K,\FCal)=0$ for $i>2n.$
 \item For each $i\in\Z,$ there is a canonical homomorphism $\HH_{\rm rig,c}^{i}(X\slash K,\FCal)\rightarrow \HH_{\rm rig}^{i}(X\slash K,\FCal)$ which is an isomorphism if $X$ is proper.
 \item Let $Z\subset X$ be a closed subvariety and let $U=X\setminus Z$ be its open complement. Then there is a long exact sequence 
 \begin{equation*}
  \ldots \rightarrow \HH_{\rm rig,c}^{i-1}(Z\slash K,\FCal)\rightarrow \HH_{\rm rig,c}^{i}(U\slash K,\FCal)\rightarrow \HH_{\rm rig,c}^{i}(X\slash K,\FCal)\rightarrow \HH_{\rm rig,c}^{i}(Z\slash K,\FCal)\rightarrow\ldots .
 \end{equation*}
 \item If $X$ is of dimension $n,$ then for each $i\in \{0,\ldots,2n\},$ there is a perfect pairing
 $$\HH_{\rm rig}^{i}(X\slash K,\FCal)\times\HH_{\rm rig,c}^{2n-i}(X\slash K,\FCal^{\vee})\rightarrow K(-n)[-2n]$$
 of $K$-spaces, where $\FCal^{\vee}$ is the isocrystal dual to $\FCal$ (defined in the usual way). Here and in the sequel, $(-)(l)$ means that the Frobenius automorphism $F$ acts by multiplication with $q^{-l}$ (``$l$-th Tate twist'') and $(-)[j]$ means ``shift in degree $-j$''. For an arbitrary $K$-vector space $V$ and an integer $i\in\Z,$ set 
 $$V(i)=V\otimes_KK(i),$$
 i.e.\ $F$ acts on this space through its action on $K(i).$
 \item If $X$ is affine of dimension $n,$ then $\HH_{\rm rig}^{i}(X\slash K,\FCal)=0$ for $i\notin\{0,\ldots,n\}.$ 
 \item Let $n\in\N_0.$ Then there are the following identifications:
 \begin{eqnarray}\label{Brucke} 
  \HH_{\rm rig}^{\ast}(\PR_k^n\slash K) &=& \bigoplus_{i=0}^n K(-i)[-2i],\nonumber\\
  \HH_{\rm rig,c}^{\ast}(\A_k^n\slash K)&=& K(-n)[-2n],\\
  \HH_{\rm rig}^{\ast}(\A_k^n\slash K)&=& K(0)[0].\nonumber
 \end{eqnarray}
\end{itemize}

\section{Rigid Cohomology of \texorpdfstring{$\Drin$}{X} Computed Indirectly}\label{SecThreeThree}

In this section, the rigid cohomology $\HH_{\rm rig}^{\ast}(\CDrin\slash K)=\HH_{\rm rig,c}^{\ast}(\CDrin\slash K)$ of the closed complement $\CDrin$ of $\Drin$ in $\PR_k^n$ is computed directly as a hyper\-co\-ho\-mo\-lo\-gy. Then, the long exact sequence for rigid cohomology with compact supports for the pair of inclusions 
$$\Drin\stackrel{\text{open}}{\hookrightarrow}\PR_k^n\stackrel{\text{closed}}{\hookleftarrow}\CDrin$$ 
is used to determine $\HH_{\rm rig,c}^{\ast}(\Drin\slash K).$ The rigid cohomology of $\Drin$ is then known via Poincar\'e duality.

\subsection{Adaption of Orlik's Complex}\label{SecThreeTwo}

In order to compute $\HH_{\rm rig}^{\ast}(\CDrin\slash K)$, there will be made use of the following modified version of Orlik's complex:

For a proper subset $I\subsetneq\Delta$ with $i=i(I)=\min\{j\in\N_0\mid\alpha_{j}\in\Delta\setminus I\},$ let 
$$Y_I=\PR_k^{i(I)}.$$ 
Each inclusion $I\subset J\subsetneq\Delta$ then induces closed embeddings
$$\iota_{I,J}:Y_I\hookrightarrow Y_J$$
and for any two $I,J\subsetneq \Delta$ the identity $Y_{I\cap J}=Y_I\cap Y_J$ holds. The variety $Y_I$ is stabilized by the parabolic subgroup $\Parbf_{I,k}$ under the action of $\G_k$ on $\PR_k^n.$ By construction, there is then an identification
$$\CDrin=\bigcup_{I\subsetneq\Delta}\bigcup_{g\in G\slash P_I}g. Y_I.$$
For $I\subsetneq\Delta$ and $g\in G\slash P_I$ write 
$$\Phi_{g,I}:g.Y_I\hookrightarrow\CDrin$$
for the closed embedding given by inclusion. If furthermore $J\subsetneq\Delta$ and $h\in G\slash P_J$ with $I\subset J$ and such that $g P_I$ is mapped to $h P_J$ under the canonical map $ G\slash P_I\rightarrow  G\slash P_J,$ then write
$$\iota_{I,J}^{g,h}:g.Y_I\hookrightarrow h.Y_J$$
for the closed embedding given by inclusion.\\

Now associate to each object its tube and then the respective adic space, i.e.\ there are adic spaces $]\CDrin[^{\sad}$ and  $]g.Y_I[^{\sad}$ for $I\subsetneq \Delta,$ $g\in G,$ all of which are open in $\PR_K^{n,\sad}.$ For $I\subset J\subsetneq\Delta$ and $g,h\in G$ with $g P_I\mapsto h P_J$ under $G\slash P_I\rightarrow G\slash P_J,$ the maps $\iota_{I,J}^{g,h}$ and $\Phi_{g,I}$ give rise to open embeddings of adic spaces associated with the respective tubes 
\begin{eqnarray*}
  ]g.Y_I[^{\sad} &\hookrightarrow& ]h.Y_J[^{\sad},\\
  ]g.Y_I[^{\sad} &\hookrightarrow& ]\CDrin[^{\sad}.
\end{eqnarray*}

For cohomological purposes, it will be more convenient to work with closed embeddings. Therefore, denote by 
$$\spec^{\sad}:\PR_K^{n,\sad}\rightarrow\PR_k^n$$ 
the adic specialization map. It is continuous, cf.\ \cite[1.9]{HuberBook} resp.\ \cite[Section 4]{HuberMZ}, and for a locally closed subvariety $X\subset\PR_k^n,$ the adification $]X[_P^{\sad}$ of its rigid-analytic tube identifies with the interior $(\spec^{\sad})^{-1}(X)^o$ of its adic tube, cf.\ \cite[Lemma 5.6.9]{HuberBook}. For each $I\subset J$ and $gP_I\mapsto h P_J$ as above, there is then a commutative triangle 

\begin{equation}\label{SecDiag}
\begin{diagram}
  &  & (\spec^{\sad})^{-1}(\CDrin) & \\
  & \ruInto^{\Phi_{g,I}^{\sad}}  & \uInto_{\Phi_{h,J}^{\sad}}\\
  (\spec^{\sad})^{-1}(g.Y_I) & \rInto^{\iota_{I,J}^{g,h,\sad}} & (\spec^{\sad})^{-1}(h.Y_J)\\
  &  
\end{diagram} 
\end{equation}

\noindent of closed embeddings of closed subspaces of $(\spec^{\sad})^{-1}(\CDrin).$ Here, the maps appearing in the triangle are again induced by the respective maps of $k$-varieties above. Let $\ECal$ be a sheaf of abelian groups on $(\spec^{\sad})^{-1}(\CDrin).$ Each triangle (\ref{SecDiag}) gives rises to a morphism 
$$p_{I,J}^{g,h}:(\Phi_{h,J}^{\sad})_*(\Phi_{h,J}^{\sad})^*\ECal\rightarrow (\Phi_{g,I}^{\sad})_*(\Phi_{g,I}^{\sad})^*\ECal$$ 
of sheaves on $(\spec^{\sad})^{-1}(\CDrin)$ via the adjunction property of the involved functors. Define
\begin{eqnarray*}
\ECal_{g,I} &=& \left(\Phi_{g,I}\right)_*\left(\Phi_{g,I}\right)^*\ECal,\nonumber\\ 
 \ECal_{I}  &=& \bigoplus_{g\in G\slash P_I}\ECal_{g,I},\nonumber
\end{eqnarray*}
and set 
\begin{eqnarray*}
p_{I,J} &=& \bigoplus_{(g,h)\in G\slash P_I\times G\slash P_J}p_{I,J}^{g,h},\nonumber\\
d_{I,J} &=& \begin{cases} (-1)^{i}p_{I,J} &\text{if }J=I\dot\cup\{\alpha_i\}\\
                                          0 &\text{else},
\end{cases}\nonumber
\end{eqnarray*}
where $p_{I,J}^{g,h}=0$ if $gP_I$ does not map to $h P_J.$ The maps $d_{I,J}$ now induce differentials so that the following complex of sheaves on $(\spec^{\sad})^{-1}(\CDrin)$ is defined: 
\begin{equation}\label{OrlikFundKom}
 0\rightarrow\ECal\rightarrow\bigoplus_{I\subsetneq\Delta\atop \# I=n-1}\ECal_I\rightarrow \ldots \rightarrow\bigoplus_{I\subsetneq\Delta\atop \# I=1}\ECal_I\rightarrow \ECal_{\emptyset} \rightarrow 0. 
\end{equation}

\begin{prop}{\rm (cf.\ \cite[Satz 5.3]{OrlikDiss})}\label{FundaKompAd}
The complex (\ref{OrlikFundKom}) is acyclic.
\end{prop}
\begin{proof}
 For $i\in\{0,\ldots,n-1\},$ set 
 $$\ECal_i=\bigoplus_{I\subsetneq\Delta\atop\# I=n-1-i}\ECal_I.$$ 
 As was mentioned in the introduction to this chapter, sheaves on adic spaces are in particular sheaves on topological spaces. Therefore, to prove the proposition, it is sufficient to check that for each point $x\in(\spec^{\sad})^{-1}(\CDrin),$ the localized complex 
 \begin{equation*}
  0\rightarrow (\ECal)_{x}\rightarrow (\ECal_{0})_{x}\rightarrow (\ECal_{1})_{x}\rightarrow\ldots\rightarrow (\ECal_{n-1})_{x}\rightarrow 0\nonumber
 \end{equation*}
 is acyclic. By construction, this last complex is equal to
 \begin{equation}\label{ChainFKom}
  0\rightarrow\ECal_{x} \rightarrow\bigoplus_{I\subsetneq\Delta\atop\# I =n-1}\bigoplus_{g\in G\slash P_I\atop x\in (\spec^{\sad})^{-1}(g.Y_I)}\ECal_{x}\rightarrow\ldots\rightarrow  \bigoplus_{I\subsetneq\Delta\atop\# I =1}\bigoplus_{g\in G\slash P_I\atop x\in (\spec^{\sad})^{-1}(g.Y_I)}\ECal_{x}\rightarrow \bigoplus_{g\in G\slash\Parbf_{\emptyset}(k)\atop x\in(\spec^{\sad})^{-1}(g.Y_{\emptyset})}\ECal_{x}\rightarrow 0.
 \end{equation}
 Let 
 $$X=\left\{g P_I\in G\slash P_I\text{ }\big |\text{ } \# I =n-1,x\in (\spec^{\sad})^{-1}(g.Y_I)\right\}$$
 and equip $X$ with a partial order structure ``$\leq$'' by identifying
 $$\coprod_{I\subset\Delta\atop \# I=n-1} G\slash P_I$$
 with 
 $$\left\{U\subset k^{n+1}\text{ }k\text{-subspace}\mid U\neq (0),k^{n+1}\right\}$$
 and its natural partial order structure given by inclusion. Then $X$ is identified with the set
 $$\left\{U\subset k^{n+1}\text{ }k\text{-subspace}\text{ }\big |\text{ } U\neq (0),k^{n+1}\text{ and }x\in  (\spec^{\sad})^{-1}(\PR(U))\right\}.$$
 Furthermore, $X$ is nonempty since $(\spec^{\sad})^{-1}(\CDrin)$ is covered by the union of all $(\spec^{\sad})^{-1}(\PR(U))$ with $U$ running through all proper subspaces of $k^{n+1}.$ 
 
 For $i\in\{0,\ldots,n-1\},$ let 
 $$X^i=\{(x_0,\ldots,x_i)\mid\forall j\in\{0,\ldots,i\}: x_j\in X, x_0< x_1<\ldots< x_i\}.$$ 
 By construction, $X^{\bullet}=\bigcup_{i=0}^{n-1}X^i$ 
 has the structure of a simplicial complex with $X^i$ the set of $i$-simplices and (\ref{ChainFKom}) is the chain complex with values in $\ECal_{x}$ associated with this simplicial complex.
 
 Let $U_0\in X$ be a subspace of minimal dimension and let $U\in X$ be arbitrary. It follows that 
 $$x \in (\spec^{\sad})^{-1}(\PR(U_0))\cap (\spec^{\sad})^{-1}(\PR(U))= (\spec^{\sad})^{-1}(\PR(U_0)\cap \PR(U))=(\spec^{\sad})^{-1}(\PR(U_0\cap U)).$$
 so that $U\cap U_0\neq (0)$ (as otherwise $\PR(U\cap U_0)$ would be empty). By minimality of $U_0,$ this implies that $U_0\cap U= U_0.$ Therefore, the identity map 
 $$\id:X\rightarrow X, U\mapsto U,$$
 fulfills
 $$U_0,U\leq \id(U)$$
 for all $U\in X$ and by Quillen's criterion (cf.\ \cite[1.5]{Quillen}), the complex (\ref{ChainFKom}) is then acyclic. This proves the proposition.
\end{proof}
 
\begin{cor}\label{FundaKompRevArrCorr}
 Let $\ECal^{\bullet}$ be a finite complex of sheaves of abelian groups on $(\spec^{\sad})^{-1}(\CDrin).$ Then, the complex 
 \begin{eqnarray}\label{FDKomp}
  0\rightarrow\ECal^{\bullet}\rightarrow\bigoplus_{I\subsetneq\Delta\atop \# I=n-1}\ECal_I^{\bullet}\rightarrow \ldots \rightarrow\bigoplus_{I\subsetneq\Delta\atop \# I=1}\ECal_I^{\bullet}\rightarrow \ECal_{\emptyset}^{\bullet}\rightarrow  0
 \end{eqnarray}
 is acyclic, i.e.\ (\ref{FDKomp}) is an exact sequence in the category of complexes of sheaves of abelian groups on $(\spec^{\sad})^{-1}(\CDrin).$
\end{cor}

\subsection{Construction of a Spectral Sequence}\label{ThreeThreeOne}

Denote by 
$$f:\text{ }]\CDrin[^{\sad}=(\spec^{\sad})^{-1}(\CDrin)^o\hookrightarrow (\spec^{\sad})^{-1}(\CDrin)$$ 
the canonical open immersion of the interior of $(\spec^{\sad})^{-1}(\CDrin).$
For the rest of this section, let $$\ECal^{\bullet}=\Om_{(\spec^{\sad})^{-1}(\CDrin)}\rightarrow\Omega^1_{(\spec^{\sad})^{-1}(\CDrin)}\rightarrow\Omega^2_{(\spec^{\sad})^{-1}(\CDrin)}\rightarrow\ldots\rightarrow\Omega^n_{(\spec^{\sad})^{-1}(\CDrin)}$$ 
be the de Rham complex on $(\spec^{\sad})^{-1}(\CDrin).$ Plug this complex into (\ref{FDKomp}) and set 
\begin{eqnarray*}
 \ECal_{-1}^{\bullet} &=& \ECal^{\bullet},\\
 \ECal_i^{\bullet} &=& \bigoplus_{I\subsetneq\Delta\atop \#I=n-1-i}\ECal_I^{\bullet}
\end{eqnarray*}
for $i=0,\ldots,n-1.$  The functor $f^*$ is exact and therefore, the complex
$$0\rightarrow f^*\ECal_{-1}^{\bullet}\rightarrow f^*\ECal_{0}^{\bullet}\rightarrow f^*\ECal_{1}^{\bullet}\rightarrow \ldots\rightarrow f^*\ECal_{n-1}^{\bullet}\rightarrow  0$$
(consisting of complexes of sheaves on $]\CDrin[^{\sad}$) is still acyclic. Application of the Godement functor $\GCal$ (cf.\ e.g.\ \cite[6.72-73]{Rotman}) to each sheaf appearing in (\ref{FDKomp}) plus application of the global sections functor induces a first quadrant double complex of $K$-vector spaces

\begin{diagram}
 \Gamma(]\CDrin[^{\sad},\GCal(f^*\ECal_{0}^n))     & \rTo & \Gamma(]\CDrin[^{\sad},\GCal(f^*\ECal_{1}^n))     & \rTo & \ldots & \rTo    & \Gamma(]\CDrin[^{\sad},\GCal(f^*\ECal_{n-1}^n)) \\
 \uTo                                         &     & \uTo&&& & \uTo \\
 \vdots                                       &     &      \vdots                                         &        &   \vdots &&      \vdots\\ 
 \uTo     &&\uTo                                    &     &                                           &                      & \uTo \\
 \Gamma(]\CDrin[^{\sad},\GCal(f^*\ECal_{0}^1))     & \rTo    &     \Gamma(]\CDrin[^{\sad},\GCal(f^*\ECal_{1}^1))     & \rTo &\ldots                                         & \rTo       &          \Gamma(]\CDrin[^{\sad},\GCal(f^*\ECal_{n-1}^1))\\ 
 \uTo         &      &  \uTo  &&                                         &                      & \uTo \\
 \Gamma(]\CDrin[^{\sad},\GCal(f^*\ECal_{0}^0))     & \rTo & \Gamma(]\CDrin[^{\sad},\GCal(f^*\ECal_{1}^0)) &\rTo & \ldots & \rTo    & \Gamma(]\CDrin[^{\sad},\GCal(f^*\ECal_{n-1}^0)) 
\end{diagram}

\bigskip
\noindent with exact rows.

There is then canonically a quasi-isomorphism of the complex $\Gamma(]\CDrin[^{\sad},\GCal(f^*\ECal^{\bullet}))$ into the associated total complex of the above double complex which implies the existence of a spectral sequence
\begin{eqnarray*}
E_1^{r,s}=h^s(\Gamma(]\CDrin[^{\sad},\GCal(f^*\ECal_r^{\bullet})))&\Longrightarrow&h^{r+s}(\Gamma(]\CDrin[^{\sad},\Tot(\GCal(f^*\ECal_{\bullet}^{\bullet}))))\\ 
&=& h^{r+s}(\Gamma(]\CDrin[^{\sad},\GCal(f^*\ECal^{\bullet})))\\
&=& h^{r+s}(\Gamma(]\CDrin[^{\sad},f^*\GCal(\ECal^{\bullet})))\\
&=& h^{r+s}(\Gamma(]\CDrin[^{\sad},\GCal(\ECal^{\bullet}))).
\end{eqnarray*}

It follows from the fact that hypercohomology can be computed from flasque resolutions (cf.\ \cite[Appendix]{EsVie}) and from the fact that $\CDrin$ is closed in $\PR_k^n,$ that
$$h^{\ast}(\Gamma(]\CDrin[^{\sad},\GCal(\ECal^{\bullet})))=\HH_{\rm rig}^{\ast}(\CDrin\slash K),$$ 
see also Subsection \ref{BerDefRiCoh}.

\subsection{Evaluation of the Spectral Sequence}\label{ThreeThreeTwo}

\subsubsection{The \texorpdfstring{$E_1$}{E1}-Page}

\begin{lemma}\label{FirstRigE1Lem} Let $r,s\in\Z.$ Then there is an identification
 $$E_1^{r,s}=\bigoplus_{I\subsetneq\Delta\atop \#I=n-1-r}\Ind_{ P_I}^{ G}\HH_{\rm rig}^s(Y_I\slash K).$$
\end{lemma}
\begin{proof}
Using compatibility of Godement resolution with direct and inverse images of sheaves in the present situation, one calculates as follows:
 {\allowdisplaybreaks \begin{eqnarray*}
  E_1^{r,s} &=& h^{s}\left(\Gamma\left(]\CDrin[^{\sad},\GCal\left(f^*\ECal_r^{\bullet}\right)\right)\right)\nonumber\\
  &=& h^{s}\left(\Gamma\left(]\CDrin[^{\sad},\GCal\left(\ECal_r^{\bullet}\right)\right)\right)\nonumber\\
  &=& h^{s}\left(\Gamma\left(]\CDrin[^{\sad},\GCal\left(\bigoplus_{I\subsetneq\Delta,\atop \#I=n-1-r}\bigoplus_{g\in   G\slash  P_I}(\Phi_{g,I}^{\sad})_*(\Phi_{g,I}^{\sad})^*\ECal^{\bullet}\right)\right)\right)\nonumber\\
  &=& \bigoplus_{I\subsetneq\Delta,\atop \#I=n-1-r}\bigoplus_{g\in   G\slash  P_I}h^{s}\left(\Gamma\left(]\CDrin[^{\sad},\GCal\left((\Phi_{g,I}^{\sad})_*(\Phi_{g,I}^{\sad})^*\ECal^{\bullet}\right)\right)\right)\nonumber\\
  &=& \bigoplus_{I\subsetneq\Delta,\atop \#I=n-1-r}\bigoplus_{g\in   G\slash  P_I}h^{s}\left(\Gamma\left(]\CDrin[^{\sad},(\Phi_{g,I}^{\sad})_*\GCal\left((\Phi_{g,I}^{\sad})^*\ECal^{\bullet}\right)\right)\right)\nonumber\\
  &=& \bigoplus_{I\subsetneq\Delta,\atop \#I=n-1-r}\bigoplus_{g\in   G\slash  P_I}h^{s}\left(\Gamma\left(]\CDrin[^{\sad}\cap(\spec^{\sad})^{-1}(g.Y_I),\GCal\left((\Phi_{g,I}^{\sad})^*\ECal^{\bullet}\right)\right)\right).\nonumber
  \end{eqnarray*}}
  By construction, 
  $$]\CDrin[^{\sad}\cap(\spec^{\sad})^{-1}(g.Y_I)=]g.Y_I[^{\sad}$$
  for all $I\subsetneq\Delta$ and all $g\in  G\slash P_I.$ It then follows that 
  \begin{eqnarray*}
   E_1^{r,s} &=& \bigoplus_{I\subsetneq\Delta,\atop \#I=n-1-r}\bigoplus_{g\in   G\slash  P_I}h^{s}\left(\Gamma\left(]g.Y_I[^{\sad},\GCal\left((\Phi_{g,I}^{\sad})^*\ECal^{\bullet}\right)\right)\right).\nonumber\\
   &=& \bigoplus_{I\subsetneq\Delta,\atop \#I=n-1-r}\bigoplus_{g\in   G\slash  P_I}\Hyp^{s}\left(]g.Y_I[^{\sad},\ECal_{\mid ]g.Y_I[^{\sad}}^{\bullet}\right)\\
   &=& \bigoplus_{I\subsetneq\Delta,\atop \#I=n-1-r}\bigoplus_{g\in   G\slash  P_I}\Hyp^{s}\left(]g.Y_I[,\ECal_{\mid ]g.Y_I[}^{\bullet}\right).\nonumber
  \end{eqnarray*}
 For $g\in\G(\VCal)$ also denote by $g$ its induced automorphism on $\PR_K^{n,\rm rig}$ via the action
 $$\G(\VCal) \times \PR_K^{n,\rm rig} \rightarrow\PR_K^{n,\rm rig}, (g,x)\mapsto g.x=xg^{-1}.$$
 Then the diagram
 \begin{diagram}
  \PR_K^{n,\rm rig} & \rTo^{g} & \PR_K^{n,\rm rig}\\
  \dTo^{\spec^{\rm rig}} & & \dTo_{\spec^{\rm rig}}\\
  \PR_k^n & \rTo^{\overline g} & \PR_k^n
 \end{diagram}
 commutes, where $\overline g$ is the image of $g$ under the canonical map $\G(\VCal)\rightarrow G.$ Therefore,
 \begin{eqnarray*}
  g.]Y_I[ &=& g.\left\{x\in\PR_K^{n,\rm rig}\text{ }\big |\text{ } \spec^{\rm rig}(x)\in Y_I\right\}\\
  &=& \left\{g.x\in\PR_K^{n,\rm rig}\text{ }\big |\text{ } \spec^{\rm rig}(x)\in Y_I\right\}\\
  &=& \left\{x\in\PR_K^{n,\rm rig}\text{ }\big |\text{ } \spec^{\rm rig}(g^{-1}.x)\in Y_I\right\}\\
  &=& \left\{x\in\PR_K^{n,\rm rig}\text{ }\big |\text{ } (\overline g)^{-1}.\spec^{\rm rig}(x)\in Y_I\right\}\\
  &=& \left\{x\in\PR_K^{n,\rm rig}\text{ }\big |\text{ } \spec^{\rm rig}(x)\in \overline g.Y_I\right\}\\
  &=& ]\overline g.Y_I[.
 \end{eqnarray*}
 Define $\Parbf_I(\VCal)(1)$ to be the preimage of $ P_I$ with respect to the above map $\G(\VCal)\rightarrow G.$ Then $\Parbf_I(\VCal)(1)$ stabilizes $]Y_I[.$ For $g\in\G(\VCal)\slash\Parbf_I(\VCal)(1),$ set
 $$g.\Hyp^s(]Y_I[,\ECal_{\mid]Y_I[}^\bullet)=\Hyp^s(g.]Y_I[,\ECal_{\mid g.]Y_I[}^\bullet).$$
 This is well-defined and from the identity above, since $\G(\VCal)\slash \Parbf_I(\VCal)(1)\cong  G\slash P_I,$ it then follows that
 \begin{eqnarray*}
 E_1^{r,s} &=&  \bigoplus_{I\subsetneq\Delta,\atop \#I=n-1-r}\bigoplus_{g\in  \G(\VCal)\slash \Parbf_I(\VCal)(1)}\Hyp^{s}\left(g.]Y_I[,\ECal_{\mid g.]Y_I[}^{\bullet}\right)\\
 &=&  \bigoplus_{I\subsetneq\Delta,\atop \#I=n-1-r}\bigoplus_{g\in  \G(\VCal)\slash \Parbf_I(\VCal)(1)}g.\Hyp^{s}\left(]Y_I[,\ECal_{\mid ]Y_I[}^{\bullet}\right)\\
 &=& \bigoplus_{I\subsetneq\Delta,\atop \#I=n-1-r}\bigoplus_{g\in   G\slash  P_I}g.\HH_{\rm rig}^{s}\left(Y_I\slash K\right)\nonumber\\
 &=& \bigoplus_{I\subsetneq\Delta,\atop \#I=n-1-r}\Ind_{ P_I}^{ G}\HH_{\rm rig}^{s}\left(Y_I\slash K\right).\nonumber
 \end{eqnarray*}
 The last identity holds since, by functoriality, $\HH_{\rm rig}^s(Y_I\slash K)$ is a $ P_I$-module and as a $K$-vector space, $g.\HH_{\rm rig}^{s}\left(Y_I\slash K\right)$ is isomorphic to $\HH_{\rm rig}^s(Y_I\slash K).$ This finishes the proof.
\end{proof}

Directly from the definition, the fact that $Y_I\cong\PR_k^{i_1}$ for $i_1\in\{0,\ldots, n-1\}$ minimal with $\alpha_{i_1}\in\Delta\setminus I,$ gives $$\HH_{\rm rig}^{\ast}(Y_I\slash K)=\bigoplus_{j=0}^{i_1} K(-j)[-2j],$$ 
see (\ref{Brucke}). For $j\in\{0,\ldots,n-1\},$ set 
$$I_j=\begin{cases}
 \{\alpha_0,\ldots,\alpha_{j-1}\} &\text{ if }j\in\{1,\ldots,n-1\}\\
 \emptyset &\text{ if }j=0.\end{cases}$$ Then, a row $E_1^{\bullet,s}$ of the first page of the spectral sequence has non-zero entries only for even $s\in\{0,2,\ldots,2n-2\}$ and reads
\begin{eqnarray*}
E_1^{\bullet,s}&:& \text{ }E_1^{ 0,s}=\bigoplus_{I_{\frac s 2}\subset I\subsetneq\Delta\atop\#I= n-1}\Ind_{ P_I}^{ G}K(-\frac s 2)\rightarrow \ldots\rightarrow\bigoplus_{I_{\frac s 2}\subset I\subsetneq\Delta\atop\#I=\frac s 2 +1}\Ind_{ P_I}^{ G}K(-\frac s 2)\\ &&\rightarrow E_1^{n-1-\frac s 2,s}=\Ind_{P_{I_{\frac s 2}}}^{ G}K(-\frac s 2).  
\end{eqnarray*}

\subsubsection{The \texorpdfstring{$E_2$}{E2}-Page}

For each proper subset $J\subsetneq\Delta$, the sequence 
\begin{equation}\label{SteinRes}
0\rightarrow K\rightarrow\bigoplus_{J\subset I\subsetneq\Delta\atop\# (\Delta\setminus I)=1}\Ind_{ P_I}^{ G}K\rightarrow\ldots\rightarrow\bigoplus_{J\subset I\subsetneq\Delta\atop \# (\Delta\setminus I)=n-1-\#J}\Ind_{ P_I}^{ G}K\rightarrow v_{ P_J}^{ G}(K)\rightarrow 0
\end{equation}
of $ G$-modules is exact, see e.g. \cite[3.2.5]{DOR}. Therefore, the $E_2$-terms of the spectral sequence can be read off:
\begin{lemma}\label{ETO} Let $r,s\in\Z.$ Then there is an identification
 \begin{equation*}
  E_2^{r,s}=h^{r}(E^{\bullet,s})=\begin{cases}
            {v_{P_{I_{\frac s 2}}}^{ G}}(K)(-\frac s 2) &\text{ if }s\in\{0,2,\ldots,2(n-2)\}, r=n-1-\frac s 2\\
            K(-\frac s 2)&\text{ if }s\in\{0,2,\ldots,2(n-2)\}, r=0\\
            \Ind_{P_{I_{n-1}}}^{ G}K(-(n-1)) &\text{ if }s=2n-2,r=0\\
            0 &\text{ else.}
           \end{cases}\nonumber
\end{equation*}
\end{lemma}
Taking into account the fact that there are no nontrivial homomorphisms of modules with different Tate twist, one can now read off from the shape of the $E_2$-page that all resulting morphisms in $E_m,$ $m\geq 2,$ are necessarily trivial and therefore, the spectral sequence degenerates in the $E_2$-page.

\subsection{Computation of the Rigid Cohomology Modules}\label{ThreeThreeThree}

Evaluation of the filtration associated with the spectral sequence now yields the rigid cohomology of $\CDrin.$ 

\begin{prop}\label{RigCohCDrinComp}
 The rigid cohomology modules of $\CDrin$ have the following shape:
 $$\HH_{\rm rig}^s(\CDrin\slash K)=\begin{cases}
                                                  K(-\frac s 2) &\text{ if }s\in\{0,\ldots,n-2\}\text{ even}\\
                                                  K(-\frac s 2)\oplus {v_{P_{I_{n-1-s}}}^{ G}}(K)(n-1-s) &\text{ if }s\in\{n-1,\ldots,2n-3\}\text{ even}\\
                                                  {v_{P_{I_{1+s-n}}}^{ G}}(K)(n-1-s) &\text{ if }s\in\{n-1,\ldots,2n-3\}\text{ odd}\\
                                                  \Ind_{P_{I_{n-1}}}^{ G}(K)(-(n-1)) &\text{ if }s=2n-2\\
                                                  0 &\text{ else}.
                                                 \end{cases}
$$
\end{prop}

\begin{proof}
Recall that
$$h^{\ast}(\Gamma(]\CDrin[^{\sad},\GCal(\ECal^{\bullet})))=\HH_{\rm rig}^{\ast}(\CDrin\slash K).$$ 
By construction, $E_2=E_{\infty}$ now describes steps of descending filtrations of $ G$-modules on each $\HH_{\rm rig}^{s}(\CDrin\slash K),$ $s\in\N_0,$ via 
$$E_2^{r,s}=\gr^r(\HH_{\rm rig}^{r+s}(\CDrin\slash K)).$$
There are thus descending filtrations on each $\HH_{\rm rig}^s(\CDrin\slash K)$ with filtration steps
$$\gr^r(\HH_{\rm rig}^s(\CDrin\slash K))=E_2^{r,s-r}=\begin{cases}
            {v_{P_{I_{\frac {s-r} 2}}}^{ G}}(K)(-\frac {s-r} 2) &\text{if }s-r\in\{0,2,\ldots,2(n-2)\},\\ &\text{ }r=n-1-\frac {s-r} 2\\
            K(-\frac {s-r} 2)&\text{if }s-r\in\{0,2,\ldots,2(n-2)\},\\ &\text{ }r=0\\
            \Ind_{P_{I_{n-1}}}^{ G}K(-(n-1)) &\text{if }s-r=2n-2,r=0\\
            0 &\text{else.}
           \end{cases}\nonumber$$
This shows that for each $s\in\{0,\ldots,2n-2\},$ the associated graded module of the rigid cohomology module $\HH_{\rm rig}^s(\CDrin\slash K)$ has the following shape:
$$\gr^{\bullet}(\HH_{\rm rig}^s(\CDrin\slash K))=\begin{cases}
                                                  K(-\frac s 2) &\text{\rm if }s\in\{0,\ldots,n-2\}\text{ even}\\
                                                  K(-\frac s 2)\oplus {v_{P_{I_{n-1-s}}}^{ G}}(K)(n-1-s) &\text{\rm if }s\in\{n-1,\ldots,2n-3\}\text{ even}\\
                                                  {v_{P_{I_{-n+1+s}}}^{ G}}(K)(n-1-s) &\text{\rm if }s\in\{n-1,\ldots,2n-3\}\text{ odd}\\
                                                  \Ind_{P_{I_{n-1}}}^{ G}(-(n-1)) &\text{\rm if }s=2n-2.
                                                 \end{cases}
$$
Each filtration splits as it is well-known that there are no non-split extensions between modules of different Tate twist and therefore, 
$$\gr^{\bullet}(\HH_{\rm rig}^s(\CDrin\slash K))\cong \HH_{\rm rig}^s(\CDrin\slash K)=\HH_{\rm rig,c}^s(\CDrin\slash K).$$
\end{proof}

\begin{thm}\label{RigDrinFirstComp}
 The rigid cohomology with compact supports of $\Drin$ is given by 
 $$\HH_{\rm rig,c}^{\ast}(\Drin\slash K)=\bigoplus_{i=0}^n{v_{P_{I_i}}^{ G}}(K)(-i)[-n-i].$$
\end{thm}
\begin{proof}
Using the respective property from Subsection \ref{RigProps}, it follows from the fact that $\Drin$ is a smooth affine variety of dimension $n$ that $\HH_{\rm rig,c}^i(\Drin\slash K)=0$ for all $i\notin\{n,\ldots,2n\}.$ Now employ the long exact sequence for rigid cohomology with compact supports for the pair of inclusions $$\Drin\stackrel{\text{open}}{\hookrightarrow}\PR_k^n\stackrel{\text{closed}}{\hookleftarrow}\CDrin.$$ 
For each $i\in\{n,\ldots,2n\},$ there are thus exact sequences of $ G$-modules 
\begin{equation*}
  \HH_{\rm rig}^{i-1}(\PR_k^n\slash K) \rightarrow \HH_{\rm rig}^{i-1}(\CDrin\slash K) \rightarrow \HH_{\rm rig,c}^{i}(\Drin\slash K) \rightarrow \HH_{\rm rig}^{i}(\PR_k^n\slash K).\nonumber
\end{equation*}
Inductive evaluation of those exact sequences (plugging in the result from Proposition \ref{RigCohCDrinComp} and of course the fact that rigid cohomology of $\PR_k^n$ is known (see Subsection \ref{RigProps})) finishes the proof. Here, one has to use the additional fact that $\HH_{\rm rig,c}^i(\Drin\slash K)$ is pure which means that it cannot contain submodules of different Tate twist, cf.\ e.g.\ \cite[Prop. 3.3.8]{DOR} (which only uses the fact that -- as is rigid cohomology -- $\ell$-adic cohomology is a Weil cohomology). 
\end{proof}

\section{Rigid Cohomology of \texorpdfstring{$\Drin$}{X} Computed from the Associated De Rham Complex}\label{SecThreeFour}

The goal of this section is to 
compute the rigid cohomology  $\HH_{\rm rig}^{\ast}(\Drin\slash K)$ of $\Drin$ 
from its associated de Rham complex. The main tool will again be an adapted version of Orlik's complex.\\

First of all, a cofinal family of strict open neighborhoods (with respect to the reverse inclusion ordering) of $]\Drin[$ in $\PR_K^{n,\rm rig}$ suitable for the purpose of adapting Orlik's complex has to be constructed. Let $X\subset \PR_k^n$ be an open subset and write $Z=\PR_k^n\setminus X$ for its closed complement. For $\lambda\in(0,1)\cap \lvert\overline{K}^{\times}\rvert$, let $$V^{\lambda}=\PR_K^{n,\rm rig}\setminus]Z[_{\lambda}.$$ 
Here, $]Z[_{\lambda}$ is the open tube of $Z$ of radius $\lambda$ in $\PR_K^{n,\rm rig}$ which can be described as follows, cf.\ \cite[2.3]{LeStumRigCoh}: Suppose that the vanishing ideal of $Z$ is generated by the homogeneous polynomials $f_1,\ldots,f_r\in k[T_0,\ldots,T_n].$ For each $l\in\{1,\ldots,r\},$ let $\tilde f_l\in\VCal[T_0,\ldots,T_n]$ be a (homogeneous) lift of $f_l.$ Then \label{IntReff}
$$]Z[_{\lambda}=\left\{x\in\PR_K^{n,\rm rig}\text{ (unimodular)}\text{ }\Big |\text{ } \forall l\in\{0,\ldots,r\}: \lvert \tilde f_l(x)\rvert< \lambda\right\}.$$
According to \cite[3.3.1]{LeStumRigCoh}, $V^{\lambda}$ is a strict open neighborhood of $]X[$ in $\PR_K^{n,\rm rig}.$ Moreover, the system $(V^{\lambda})_{\lambda \in (0,1)}$ is even a cofinal system of quasi-compact strict open neighborhoods of $]X[$ in $\PR_K^{n,\rm rig},$ cf.\ \cite[3.3.3]{LeStumRigCoh}. For $m\in\N,$ let 
\begin{eqnarray*}
 \lambda_m &=& \lvert\pi\rvert^{1\slash (m+1)}\in\lvert\overline K^\times\rvert.
\end{eqnarray*}
Then the countable system $(V^{\lambda_m})_{m\in\N}$ is cofinal in $(V^{\lambda})_{\lambda\in(0,1)}$ and thus it is itself a cofinal system of strict open neighborhoods of $]X[$ in $\PR_K^{n,\rm rig}.$\label{StrOpNeig}

Now specialize to the case $X=\Drin.$ A slight technical problem in adapting Orlik's complex is the fact that the ``operation'' $]-[_{\lambda}$ does not commute with taking finite unions. Therefore, there will now be constructed a cofinal system $(U^m)_{m\in\N}$ of strict open neighborhoods of $]\Drin[$ in $\PR_K^{n,\rm rig}$ better suited for the task at hand.
 
For $m\in\N,$ set 
$$U^m = \PR_K^{n,\rm rig}\setminus\bigcup_{I\subsetneq\Delta}\bigcup_{g\in G\slash P_I}]g.Y_I[_{\lambda_m}.$$

\begin{lemma}\label{LostLemma}
\hspace{2em}
 \begin{enumerate}[i)]
  \item The set $U^m$ is a strict open neighborhood of $]\Drin[$ in $\PR_K^{n,\rm rig}.$
  \item The set $U^m$ is an affinoid subvariety of $\PR_K^{n,\rm rig}.$
  \item The family $\left\{U^m\mid m\in\N\right\}$ is cofinal in the family of all strict open neighborhoods of $]\Drin[$ in $\PR_K^{n,\rm rig}.$
 \end{enumerate}
\end{lemma}
\begin{proof}
 \begin{enumerate}[i)]
  \item Since 
  $$U^m=\bigcap_{I\subsetneq\Delta}\bigcap_{g\in G\slash P_I}\PR_K^{n,\rm rig}\setminus]g.Y_I[_{\lambda_m}$$
  is a finite intersection of admissible open subsets of $\PR_K^{n,\rm rig},$ it is admissible open in $\PR_K^{n,\rm rig}$ itself. Directly from the definition of tubes, it follows that $\bigcup_{I\subsetneq\Delta}\bigcup_{g\in G\slash P_I}]g.Y_I[_{\lambda_m}$ is contained in $]\CDrin[_{\lambda_m}.$ Therefore, $V^{\lambda_m}=\PR_K^{n,\rm rig}\setminus]\CDrin[_{\lambda_m}$ is contained in $U^m$ and the claim follows from \cite[3.1.2]{LeStumRigCoh}.
  \item Denote by $\HCal$ the set of all $n$-dimensional $k$-subspaces of $k^{n+1}.$ Then there is an identification 
  \begin{eqnarray*}
   U^m &=& \PR_K^{n,\rm rig}\setminus\bigcup_{H\in\HCal}]\PR(H)[_{\lambda_m}\\ &=& \bigcap_{H\in\HCal}\PR_K^{n,\rm rig}\setminus]\PR(H)[_{\lambda_m}.
  \end{eqnarray*}
  Since a finite intersection of affinoid subvarieties of $\PR_K^{n,\rm rig}$ is again an affinoid subvariety (this is due to the fact that $\PR_K^{n,\rm rig}$ is separated, cf.\ \cite[4.10.1 and 4.3.4]{FrPu}), it is enough to show that each $\PR_K^{n,\rm rig}\setminus]\PR(H)[_{\lambda_m}$ is an affinoid subvariety of $\PR_K^{n,\rm rig}.$ Possibly after a coordinate transformation, one can reduce to the case that $H=V_+(T_n).$ Then
  \begin{eqnarray*}
   \PR_K^{n,\rm rig}\setminus]\PR(H)[_{\lambda_m} &=& \left\{x=[x_0:\ldots:x_n]\in\PR_K^{n,\rm rig}\text{ unimodular}\text{ }\big|\text{ } \lvert x_n\rvert\geq \lambda_m\right\}\\
   &\cong& \left\{\left(\frac{x_0}{x_n},\ldots,\frac{x_{n-1}}{x_n}\right)\in (\widehat{\overline K})^n\text{ }\Big|\text{ } \forall i=0,\ldots,n-1:\lvert \frac{x_i}{x_n}\rvert \leq\lambda_m^{-1}\right\}\\
   &=& \left\{(z_0,\ldots,z_{n-1})\in (\widehat{\overline K})^n\text{ }\big|\text{ } \forall i=0,\ldots,n-1:\lvert \pi z_i^{m+1} \rvert \leq 1\right\}
  \end{eqnarray*}
  and the last set is an affinoid $K$-variety.
  \item The family $(V^{\lambda})_{\lambda\in(0,1)\cap \lvert\overline K^{\times}\rvert}$ is a cofinal family of strict neighborhoods. Therefore, it is enough to show that for each $\lambda$ as above, there exists some $m\in\N$ such that $U^m\subset V^{\lambda}.$ By \cite[2.3.6]{LeStumRigCoh}, there is an admissible covering 
  $$]\CDrin[_{\lambda}\subset \bigcup_{I\subsetneq \Delta}\bigcup_{g\in G\slash P_I}]g.Y_I[_{\lambda'}$$
  for $\lambda'$ such that $$\prod_{I\subsetneq\Delta}\prod_{g\in(\G\slash\Parbf_I)(k)}\lambda'=(\lambda')^{\left(\sum_{I\subsetneq\Delta}\sum_{g\in G\slash P_I}1\right)}=\lambda.$$
  Choose $m\in\N$ large enough so that $\lambda'\leq \lambda_m.$ Then there is an admissible covering
  $$]\CDrin[_{\lambda}\subset \bigcup_{I\subsetneq \Delta}\bigcup_{g\in G\slash P_I}]g.Y_I[_{\lambda_m},$$
  hence
  $$V^{\lambda}=\PR_K^{n,\rm rig}\setminus]\CDrin[_{\lambda}\supset \PR_K^{n,\rm rig}\setminus \bigcup_{I\subsetneq \Delta}\bigcup_{g\in G\slash P_I}]g.Y_I[_{\lambda_m}=U^m$$
  which finishes the proof of the lemma.
 \end{enumerate}
\end{proof}

Denote by 
$$j_m:U^m\hookrightarrow\PR_K^{n,\rm rig}$$
the inclusion. Then the above lemma implies that the rigid cohomology of $\Drin$ with values in an overconvergent $F$-isocrystal $\ECal$ (defined on $\PR_K^{n,\rm rig}$) can be computed as
$$\HH_{\rm rig}^{\ast}(\Drin\slash K,\ECal)=\Hyp^{\ast}(\PR_K^{n,\rm rig},\dlim_{m\in\N}{j_m}_*j_m^*(\ECal\otimes_{\Om_{\PR_K^{n,\rm rig}}}\Omega_{\PR_K^{n,\rm rig}}^{\bullet})).$$ 

From now on, for $i=0,\ldots, n,$ set 
$$\ECal^i=\ECal\otimes_{\Om_{\PR_K^{n,\rm rig}}}\Omega_{\PR_K^{n,\rm rig}}^i.$$
To make use of the associated de Rham complex for the computation of the above hypercohomology, one has to calculate the cohomology spaces $\HH^{\ast}(\PR_K^{n,\rm rig},\dlim_{m\in\N}{j_m}_*j_m^*\ECal^i)$ for each $i\in\N_0.$
\begin{lemma}\label{DERHAMWORKS}
 For each $i=0,\ldots,n,$ there is an identification
 $$\HH^{\ast}(\PR_K^{n,\rm rig},\dlim_{m\in\N}{j_m}_*j_m^*\ECal^i)=\dlim_{m\in\N}\HH^0(U^m,\ECal^i)[0].$$
\end{lemma}
\begin{proof}
 First of all, applying \cite[2.3.13]{HuberBook} to the parallel situation of adic spaces yields isomorphisms
 $$\HH^{\ast}(\PR_K^{n,\rm rig},\dlim_{m\in\N}{j_m}_*j_m^*\ECal^i)=\dlim_{m\in\N}\HH^{\ast}(\PR_K^{n,\rm rig},{j_m}_*j_m^*\ECal^i).$$
 The morphism $j_m$ is quasi-Stein in the sense of \cite[p. 20]{LeStumRigCoh} which in particular implies that the higher direct images ${\bf R}^l {j_m}_*(j_m^*\ECal^i)$ vanish for $l\geq 1$, cf.\ loc.\ cit. From this, it follows that 
 $$\HH^{\ast}(\PR_K^{n,\rm rig},{j_m}_*j_m^*\ECal^i)=\HH^{\ast}(U^m,j_m^*\ECal^i )=\HH^0(U^m,j_m^*\ECal^i )=\HH^0(U^m,\ECal^i)$$
 since higher coherent cohomology on affinoid spaces vanishes by Kiehl's Theorem $B,$ cf.\ \cite[2.4]{Kiehl}. This finishes the proof. 
\end{proof}

So in essence, to compute $\HH_{\rm rig}^{\ast}(\Drin\slash K,\ECal),$ one has to compute $\HH^0 (U^m,\ECal^i)$ for all $i,$ apply the direct limit $\dlim_{m\in\N},$ plug the resulting spaces into a spectral sequence and describe the associated gradings.  This will be done in the next few subsections using the methods of \cite{OrlikEVB,OrlikDeRham} by Orlik. The strategy to determine the spaces $\HH^0(U^m,\ECal^i)$ is the same as in the previous chapter in the case of finite ground fields, namely to use local cohomology, but this time of rigid analytic spaces. For this purpose, this cohomology theory shall be recalled briefly (cf.\ \cite[1.2-3]{vdPut}):\\
Let $X$ be a rigid analytic $K$-space, let $Z\subset X$ be an admissible open subset such that $U=X\setminus Z$ is also admissible open in $X.$ For an abelian sheaf $\HCal$ on $X,$ set
$$\HH_Z^0(X,\HCal)=\ker(\HH^0(X,\HCal)\rightarrow\HH^0(U,\HCal)).$$
Then $\HH_Z^0(X,-)$ is a left exact functor and therefore it has right derived functors
$$\HH_Z^i(X,-)={\rm \bf R}^i\HH_Z^0(X,-).$$
The following hold and will be used freely in the sequel:
\begin{itemize}
 \item There is a long exact sequence 
 $$\ldots\rightarrow\HH^i(X,\HCal)\rightarrow\HH^i(U,\HCal)\rightarrow\HH_Z^{i+1}(X,\HCal)\rightarrow\HH^{i+1}(X,\HCal)\rightarrow\ldots,$$
 cf.\ \cite[1.3]{vdPut}, which also holds in the more general situation presented here.
 \item From the fact that the functor $(-)^{\sad}$ induces an equivalence of topoi combined with the fact that $\HH_Z^{\ast}(X,\HCal)$ is computed by using an injective resolution of $\HCal,$ it follows that there is an isomorphism
 \begin{equation}\label{AdRigComp}\HH_{X^{\sad}\setminus U^{\sad}}^{\ast}(X^{\sad},\HCal^{\sad})\cong \HH_{Z}^{\ast}(X,\HCal)
 \end{equation}
 where the local cohomology for adic spaces is defined as usual for topological spaces.
\end{itemize}.

Now set 
$$Y^m=\PR_K^{n,\rm rig}\setminus U^m=\bigcup_{I\subsetneq\Delta}\bigcup_{g\in G\slash P_I}]g.Y_I[_{\lambda_m}.$$ 
Then there are local cohomology groups $\HH_{Y^m}^{\ast}(\PR_K^{n,\rm rig},\ECal^{i})$ and an exact sequence
$$(0)\rightarrow \HH^{0}(\PR_K^{n,\rm rig},\ECal^{i}) \rightarrow \HH^0(U^{m},\ECal^{i})\rightarrow \HH_{Y^m}^1(\PR_K^{n,\rm rig},\ECal^{i}).$$
The groups $\HH_{Y^m}^1(\PR_K^{n,\rm rig},\ECal^{i})$ can be computed using an adapted version of Orlik's complex and one can then determine $\HH^0(U^{m},\ECal^{i}).$

Because of technical reasons concerning the localization of sheaves in points, it is again more convenient to use the framework of adic spaces.

\subsection{Adaption of Orlik's Complex}\label{OneMoreTime}

For $m\in\N,$ set
$$Z^m=\PR_K^{n,\sad}\setminus U^{m,\sad}$$
and for $I\subsetneq\Delta$ and $g\in G\slash P_I,$ set 
$$Z_{g,I}^m=\PR_K^{n,\sad}\setminus\left(\PR_K^{n,\rm rig}\setminus ]g.Y_I[_{\lambda_m}\right)^{\sad}.$$
Both $Z^m$ and $Z_{g,I}^m$ are closed subspaces in $\PR_K^{n,\sad}.$ An inclusion $I\subset J$ of proper subsets of $\Delta$ together with a mapping $g P_I\mapsto h P_J$ under the canonical map $$ G\slash P_I\rightarrow G\slash P_J$$ 
induces a closed embedding of closed subspaces
$$\gamma_{I,J}^{g,h}:Z_{g,I}^m\hookrightarrow Z_{h,J}^m$$
of $Z^m.$ Furthermore, for each $I$ and $g$ as above, there are closed embeddings
$$\delta_{g,I}:Z_{g,I}^m\hookrightarrow Z^m,$$
so that for all $g,h$ and $I,J$ as above, there are commutative triangles

\begin{diagram}
           &                              & Z^m && \\
           & \ruInto^{\delta_{g,I}}      &  & \luInto^{\delta_{h,J}} & \\
 Z_{g,I}^m &              & \rInto^{\gamma_{I,J}^{g,h}} &  & Z_{h,J}^m\\
           &
\end{diagram}

\noindent of closed embeddings. Let $\FCal$ be a sheaf of abelian groups on $Z^m.$ For $I\subsetneq\Delta$ and $g\in G\slash P_I,$ set 
\begin{eqnarray*}
 \FCal_{g,I} &=& {\delta_{g,I}}_*\delta_{g,I}^{-1}\FCal,\nonumber\\
 \FCal_I &=&  \bigoplus_{g\in G\slash P_I}\FCal_{g,I}.\nonumber
\end{eqnarray*}

\begin{prop}\label{ACDC}
 For any sheaf $\FCal$ of abelian groups on $Z^m,$ there is an acyclic complex
 \begin{equation}\label{GSeq}
  0\rightarrow\FCal\rightarrow \bigoplus_{I\subsetneq\Delta\atop \#I=n-1}\FCal_I\rightarrow \bigoplus_{I\subsetneq\Delta\atop \#I=n-2}\FCal_I\rightarrow\ldots\rightarrow \bigoplus_{I\subsetneq\Delta\atop \#I=1}\FCal_I\rightarrow\FCal_{\emptyset}\rightarrow 0
 \end{equation}
 of sheaves of abelian groups on $Z^m.$
\end{prop}
\begin{proof}
 (cf.\ the proof of Proposition \ref{FundaKompAd}) The proof is again by localization of the above complex with respect to a point $x\in Z^m.$ Consider the set
 $$X=\left\{U\subsetneq k^{n+1}\text{ }k\text{-subspace}\text{ }\big|\text{ } U\neq (0),x\in \PR_K^{n,\sad}\setminus (\PR_K^{n,\rm rig}\setminus ]\PR(U)[_{\lambda_m})^{\sad}\right\}.$$
 This set is not empty since, by construction, 
 $$Z^m=\bigcup_{(0)\subsetneq U\subsetneq k^{n+1} }\PR_K^{n,\sad}\setminus(\PR_K^{n,\rm rig}\setminus ]\PR(U)[_{\lambda_m})^{\sad}$$
 with $U$ running through all proper non-zero subspaces of $k^{n+1}.$ Choose a minimal subspace $U_0\in X$ and let $U\in X$ be arbitrary. Directly from the definition of tubes of radius $\lambda$ (see page \pageref{IntReff}), one observes that the identity
 $$]\PR(U_0)[_{\lambda_m}\cap ]\PR(U)[_{\lambda_m}=]\PR(U_0)\cap \PR(U)[_{\lambda_m}$$ 
 holds. Both $\PR_K^{n,\rm rig}\setminus]\PR(U_0)[_{\lambda_m}$ and $\PR_K^{n,\rm rig}\setminus]\PR(U)[_{\lambda_m}$ are finite unions of affinoid spaces which is seen by using the standard affinoid covering of $\PR_K^{n,\rm rig}.$ Therefore, the covering 
 $$\PR_K^{n,\rm rig}\setminus]\PR(U_0)\cap \PR(U)[_{\lambda_m} =\PR_K^{n,\rm rig}\setminus]\PR(U_0)[_{\lambda_m}\cup\text{ } \PR_K^{n,\rm rig}\setminus]\PR(U)[_{\lambda_m}$$ 
 has a refinement consisting of finitely many affinoid subsets and is thus admissible. By \cite[1.1.11 (c)]{HuberBook}, this implies that
 $$\left(\PR_K^{n,\rm rig}\setminus]\PR(U_0)\cap \PR(U)[_{\lambda_m}\right)^{\sad} =\left(\PR_K^{n,\rm rig}\setminus]\PR(U_0)[_{\lambda_m}\right)^{\sad}\cup\left(\PR_K^{n,\rm rig}\setminus]\PR(U)[_{\lambda_m}\right)^{\sad},$$
 thus
 \begin{eqnarray*}
  x &\in& \left(\PR_K^{n,\sad}\setminus \left(\PR_K^{n,\rm rig}\setminus ]\PR(U_0)[_{\lambda_m}\right)^{\sad}\right)\cap \left(\PR_K^{n,\sad}\setminus \left(\PR_K^{n,\rm rig}\setminus ]\PR(U)[_{\lambda_m}\right)^{\sad}\right)\\
  &=& \PR_K^{n,\sad}\setminus \left( \left(\PR_K^{n,\rm rig}\setminus ]\PR(U_0)[_{\lambda_m}\right)^{\sad}\cup \left(\PR_K^{n,\rm rig}\setminus ]\PR(U)[_{\lambda_m}\right)^{\sad}\right)\\
  &=& \PR_K^{n,\sad}\setminus \left( \PR_K^{n,\rm rig}\setminus ]\PR(U_0\cap U)[_{\lambda_m}\right)^{\sad}.
 \end{eqnarray*}
 It follows that $U_0\cap U$ cannot be equal to $(0)$ and, by minimality of $U_0,$ this means that $U_0\cap U=U_0.$ Therefore, the identity map $\id:X\rightarrow X$ fulfills 
 $$U_0,U\subset \id(U)$$
 for all $U\in X$ and, again by Quillen's criterion, the simplicial complex $X^{\bullet}$ associated with $X$ is contractible (cf.\ the construction in the proof of Proposition \ref{FundaKompAd}). This implies that the chain complex 
 \begin{equation*}
  0\rightarrow\FCal_{x}\rightarrow \bigoplus_{\substack {I\subsetneq\Delta\\ \#I=n-1}}\bigoplus_{\substack{g\in G\slash P_I\\ x\in Z_{g,I}^m}}\FCal_{x}\rightarrow \ldots\rightarrow \bigoplus_{\substack {I\subsetneq\Delta\\ \#I=1}}\bigoplus_{\substack{g\in G\slash P_I\\ x\in Z_{g,I}^m}} \FCal_{x}\rightarrow\bigoplus_{\substack{g\in G\slash P_\emptyset\\ x\in Z_{g,\emptyset}^m}}\FCal_{x}\rightarrow 0
 \end{equation*}
 associated with $X^{\bullet}$ with values in $X$ is acyclic. Since this complex is precisely the localization of (\ref{GSeq}) with respect to $x,$ the proposition is proved.
\end{proof}

\subsection{Construction of a Spectral Sequence}\label{ThreeFourOne}

Let $m\in\N$ and consider the case that $\FCal=\Z_{Z^m}$ is the constant sheaf on $Z^m$ with value $\Z.$ Let 
$$\iota_m:Z^m\hookrightarrow\PR_K^{n,\sad}$$ 
be the inclusion which is in particular a closed embedding. For $r=0,\ldots,n-1,$ set 
$$\FCal_r=\bigoplus_{I\subsetneq \Delta\atop \# I=n-1-r}\FCal_I.$$
There is then a (second quadrant) spectral sequence
\begin{equation*}
 E_1^{r,s}=\Ext^s({\iota_m}_*\FCal_{-r},\ECal^{i,\sad})\Longrightarrow \Ext^{r+s}({\iota_m}_*\FCal,\ECal^{i,\sad})=\HH_{Z^m}^{r+s}(\PR_K^{n,\sad},\ECal^{i,\sad}).
\end{equation*}
This spectral sequence is evaluated in the next subsection. Recall that
$$Y^m=\bigcup_{I\subsetneq\Delta}\bigcup_{g\in G\slash P_I}]g.Y_I[_{\lambda_m}.$$
It follows from (\ref{AdRigComp}) that
$$\HH_{Z^m}^{r+s}(\PR_K^{n,\sad},\ECal^{i,\sad})=\HH_{Y^m}^{r+s}(\PR_K^{n,\rm rig},\ECal^{i}).$$

\subsection{Evaluation of the Spectral Sequence}\label{ThreeFourTwo}

\subsubsection{The \texorpdfstring{$E_1$}{E1}-Page}

\begin{lemma}\label{EONELEMMA} Let $r,s\in\Z.$ Then there is an identification
$$E_1^{r,s}=\bigoplus_{I\subsetneq\Delta\atop\lvert I\rvert=n-1+r}\Ind_{\Parbf_I(\VCal)(1)}^{\G(\VCal)}\HH_{]Y_I[_{\lambda_m}}^s(\PR_K^{n,\rm rig},\ECal^{i}).$$
\end{lemma}
\begin{proof}
For brevity, write $\Z=\Z_{]\CDrin[_{\lambda_m}^{\rm ad}}.$ Then
 \begin{eqnarray}
  E_1^{r,s}&=&\Ext^s({\iota_m}_*\FCal_{-r},\ECal^{i,\sad})\nonumber\\ 
	&=& \bigoplus_{I\subsetneq \Delta\atop \lvert I\rvert=n-1+r}\bigoplus_{g\in G\slash P_I}\Ext^{s}({\iota_m}_*(\delta_{g,I})_*(\delta_{g,I})^{-1}\Z,\ECal^{i,\sad})\nonumber\\
	&=& \bigoplus_{I\subsetneq \Delta\atop \lvert I\rvert=n-1+r}\bigoplus_{g\in G\slash P_I}\Ext^{s}((\iota_m\circ\delta_{g,I})_*\Z_{\mid Z_{g,I}^m},\ECal^{i,\sad})\nonumber\\
	&=& \bigoplus_{I\subsetneq \Delta\atop \lvert I\rvert=n-1+r}\bigoplus_{g\in G\slash P_I} \HH_{Z_{g,I}^m}^{s}(\PR_K^{n,\rm ad},\ECal^{i,\sad})\nonumber\\
	&\cong&\label{Zwischen} \bigoplus_{I\subsetneq \Delta\atop \lvert I\rvert=n-1+r}\bigoplus_{g\in G\slash P_I} \HH_{]g.Y_I[_{\lambda_m}}^{s}(\PR_K^{n,\rm rig},\ECal^{i})\\
	&=&\nonumber \bigoplus_{I\subsetneq \Delta\atop \lvert I\rvert=n-1+r}\bigoplus_{g\in \G(\VCal)\slash \Parbf_I(\VCal)(1)} \HH_{g.]Y_I[_{\lambda_m}}^{s}(\PR_K^{n,\rm rig},\ECal^{i})\\
	&=& \bigoplus_{I\subsetneq\Delta\atop\lvert I\rvert=n-1+r}\Ind_{\Parbf_I(\VCal)(1)}^{\G(\VCal)}\HH_{]Y_I[_{\lambda_m}}^s(\PR_K^{n,\rm rig},\ECal^{i}),\nonumber
 \end{eqnarray}
where the isomorphism (\ref{Zwischen}) is an application of (\ref{AdRigComp}). Here, the fact that $g.]Y_I[_{\lambda_m}=]\bar g.Y_I[_{\lambda_m}$ for all $g\in\G(\VCal)$ is used to establish the action of the group $\G(\VCal)$ on $\bigoplus_{g\in G\slash P_I}\HH_{]g.Y_I[_{\lambda_m}}^s(\PR_K^{n,\rm rig},\ECal^{i}).$ For this and also for the definition of the subgroup $\Parbf_I(\VCal)(1)\subset\G(\VCal),$ cf.\ the proof of Lemma \ref{FirstRigE1Lem}. Note that the action of $\G(\VCal)$ on $\mathbb{B}_K^{n+1}=\{x=(x_0,\ldots,x_n)\in\widehat{\overline K}^{n+1}\mid\forall i=0,\ldots,n:\lvert x_i\vert\leq 1\}$ given by $(g,x)\mapsto xg^{-1}$ preserves the unimodular points of $\mathbb{B}_K^{n+1}.$
\end{proof}

By definition, $Y_I=\PR_k^{i_1}$ for $\Delta\setminus I=\{\alpha_{i_1},\ldots,\alpha_{i_t}\}$ with $i_1<\ldots<i_t$ so that 
$$\HH_{]Y_I[_{\lambda_m}}^s(\PR_K^{n,\rm rig},\ECal^{i})=\HH_{]\PR_k^{i_1}[_{\lambda_m}}^s(\PR_K^{n,\rm rig},\ECal^{i}).$$ 
In \cite[1.3]{OrlikEVB} it is shown that $\HH_{]\PR_k^{i_1}[_{\lambda_m}}^s(\PR_K^{n,\rm rig},\ECal^{i})$ carries the structure of a $K$-Banach space in such a way that the inclusion 
$$\HH_{\PR_K^{i_1,\rm rig}}^s(\PR_K^{n,\rm rig},\ECal^{i})\subset\HH_{]\PR_k^{i_1}[_{\lambda_m}}^s(\PR_K^{n,\rm rig},\ECal^{i})$$
of the algebraic local cohomology space $\HH_{\PR_K^{i_1,\rm rig}}^s(\PR_K^{n,\rm rig},\ECal^{i})$ has dense image. Thus, for $s\neq n-i_1$, the description of the local cohomology modules amounts to
\begin{equation}\label{TubeDense}
\HH_{]\PR_k^{i_1}[_{\lambda_m}}^s(\PR_K^{n,\rm rig},\ECal^{i})=\begin{cases} 0 &\text{\rm if } s<n-i_1\\ \HH^s(\PR_K^{n,\rm rig},\ECal^{i}) &\text{\rm if } s>n-i_1,
\end{cases}
\end{equation}
cf.\ the reasoning in \cite[1.2]{OrlikEVB}.
Write 
$$M_{I}^s=\begin{cases}
\HH_{]\PR_k^{n-s}[_{\lambda_m}^{\rm rig}}^s(\PR_K^{n,\rm rig},\ECal^{i}) &\text{\rm if }\alpha_{n-s}\notin I\\
\HH^s(\PR_K^{n,\rm rig},\ECal^{i}) &\text{\rm if }\alpha_{n-s}\in I
\end{cases}$$
for $I\subsetneq\Delta,s\in\{2,\ldots,n\}.$
Then, each row $E_{1}^{\bullet,s}$ with $s\in\{1,\ldots,n\}$ of the $E_1$-page of the spectral sequence has the following shape:
For $s\in\{2,\ldots,n\},$ one gets
\begin{eqnarray*}
E_1^{\bullet,s}:&& E_1^{-(s-1),s}=\Ind_{\Parbf_{\{\alpha_0,\ldots,\alpha_{n-1-s}\}}(\VCal)(1)}^{\G(\VCal)}\HH_{]\PR_k^{n-s}[_{\lambda_m}}^{s}(\PR_K^{n,\rm rig},\ECal^{i}) \rightarrow\\
&& E_1^{-(s-2),s}=\bigoplus_{I\subsetneq\Delta\atop{\# I=n-s+1\atop{\alpha_0,\ldots,\alpha_{n-1-s}\in I}}}\Ind_{\Parbf_I(\VCal)(1)}^{\G(\VCal)}M_I^s \rightarrow\ldots \rightarrow E_1^{0,s}=\bigoplus_{I\subsetneq\Delta\atop{\# I=n-1\atop{\alpha_0,\ldots,\alpha_{n-1-s}\in I}}}\Ind_{\Parbf_I(\VCal)(1)}^{\G(\VCal)(1)}M_I^s.
\end{eqnarray*}

For $s=1,$ one gets
\begin{eqnarray*}
E_1^{\bullet,1}:E_1^{0,1}=\Ind_{\Parbf_{\{\alpha_0,\ldots,\alpha_{n-2}\}}(\VCal)(1)}^{\G(\VCal)}\HH_{]\PR_k^{n-1}[_{\lambda_m}^{\rm rig}}^{1}(\PR_K^{n,\rm rig},\ECal^{i}).
\end{eqnarray*}

\subsubsection{The \texorpdfstring{$E_2$}{E2}-Page}

The evaluation of the $E_2$-page proceeds in complete analogy with \cite[2.2]{OrlikEVB}, the notable difference to the present case being the avoidance of duals. The main point -- at least in the present case -- is that for each $s\in\{2,\ldots,n\},$ the complex $E_1^{\bullet,s}$ is acyclic apart from the positions $-(s-1)$ and $0,$ cf.\ \cite[2.2.4]{OrlikEVB}. In order to avoid any (more) repetition, the proof of the following proposition is therefore omitted. The following notation is used: Let $s\in \{1,\ldots,n\}.$
\begin{itemize}
  \item The $\Parbf_{(s+1-j,s)}(\VCal)(1)$-module $\tilde\HH_{]\PR_K^s[_{\lambda_m}}^{n-s}(\PR_K^{n,\rm rig},\ECal^{i})$ is defined as
  $$\tilde\HH_{]\PR_K^s[_{\lambda_m}}^{n-s}(\PR_K^{n,\rm rig},\ECal^{i})=\ker\left(\HH_{]\PR_K^s[_{\lambda_m}}^{n-s}(\PR_K^{n,\rm rig},\ECal^{i})\xrightarrow{d}\HH^{n-s}(\PR_K^{n,\rm rig},\ECal^{i})\right),$$
  where $d$ is the ($\Parbf_{(n+1-s,s)}(\VCal)(1)$-equivariant) map appearing in the long exact sequence associated with local cohomology of rigid analytic varieties as defined in the beginning of this subsection.
  \item Denote by $\mathcal{K}_s$ the cokernel of the canonical map
  $$\bigoplus_{I\subsetneq \delta\atop {\#I=n+1-s\atop {\alpha_0,\ldots,\alpha_{n-s-1}\in I\atop \alpha_{n-s\not\in I}}}}\Ind_{\Parbf_{I}(\VCal)(1)}^{\Parbf_{(n+1-s,s)}(\VCal)(1)}(K)\rightarrow \Ind_{\Parbf_{(n+1-s,1^s)}(\VCal)(1)}^{\Parbf_{(n+1-s,s)}(\VCal)(1)}(K).$$
  \item Write $v_{\Parbf_{(n+1-s,1^s)}(\VCal)(1)}^{\G(\VCal)}(K)$ for the generalized Steinberg representation of $\G(\VCal)$ with respect to $\Parbf_{(n+1-s,1^s)}(\VCal)(1),$ i.e.\ set
  $$v_{\Parbf_{(n+1-s,1^s)}(\VCal)(1)}^{\G(\VCal)}(K)=\Ind_{\Parbf_{(n+1-s,1^s)}(\VCal)(1)}^{\G(\VCal)}(K)\slash\bigoplus_{\Parbf_{(n+1-s,1^s)}\subsetneq {\bf Q}\subset \G}\Ind_{{\bf Q}(\VCal)(1)}^{\G(\VCal)}(K).$$
 \end{itemize}

\begin{prop}\label{E2Analog}
The $E_2$-page of the above spectral sequence has the following description: 
\begin{enumerate}[i)]
 \item If $r=0$ and $s\in\{2,\ldots,n\},$ then
 $$E_2^{r,s}=\HH^s(\PR_K^{n,\rm rig},\ECal^i).$$
 \item If $s\in\{2,\ldots,n\}$ and $r=-(s-1),$ then there are short exact sequences of $ G$-modules
 \begin{eqnarray*}
  0&\rightarrow& \Ind_{\Parbf_{(n+1-s,s)}(\VCal)(1)}^{ \G(\VCal)}\left(\tilde\HH_{]\PR_k^{n-s}[_{\lambda_m}}^s(\PR_K^{n,\rm rig},\ECal^{i})\otimes\mathcal{K}_s\right)\rightarrow E_2^{-(s-1),s}\\
  &\rightarrow& v_{\Parbf_{(n+1-s,1^s)}(\VCal)(1)}^{\G(\VCal)}(K)'\otimes_K\HH^{s}(\PR_K^{n,\rm rig},\ECal^{i})\rightarrow 0
 \end{eqnarray*}
 \item In all other cases, $E_2^{r,s}=0.$
\end{enumerate}
\end{prop}

\subsubsection{Degeneration and the Resulting Filtration}

With the same arguments as in \cite[p. 633]{OrlikEVB}, the spectral sequence degenerates on its $E_2$-page and therefore, $E_2=E_{\infty}$ describes filtration steps of a (descending) filtration by $\G(\VCal)$-submodules on $\HH_{Y^m}^1(\PR_K^{n,\rm rig},\ECal^{i}).$ This filtration can be pulled back along the $\G(\VCal)$-morphism 
$$\HH^0(U^m,\ECal^{i})\rightarrow\HH_{Y^m}^1(\PR_K^{n,\rm rig},\ECal^{i}).$$
\begin{thm}\label{DescMod}
On each $\HH^0(U^m,\ECal^{i}),$ there exists a filtration 
$$\ECal^i(U^m)^{\bullet}=\left(\HH^0(U^m,\ECal^{i})=\ECal^i(U^m)^{0}\supset \ECal^i(U^{m})^{1}\supset\ldots\supset \ECal^i(U^{m})^{n-1}\supset \ECal^i(U^{m})^{n}=\HH^0(\PR_K^{n,\rm rig},\ECal^{i})\right)$$
by $\G(\VCal)$-submodules such that each filtration step appears in a short exact sequence 
\begin{eqnarray*} 
0&\rightarrow& \Ind_{\Parbf_{(j+1,n-j)}(\VCal)(1)}^{\G(\VCal)}(\tilde\HH_{]\PR_k^{j}[_{\lambda_m}}^{n-j}(\PR_K^{n,\rm rig},\ECal^{i})\otimes\mathcal{K}_{n-j}) \rightarrow  (\ECal^i(U^m)^{j}\slash\ECal^i(U^m)^{j+1})\\
&\rightarrow& v_{\Parbf_{(j+1,1^{n-j})}(\VCal)(1)}^{\G(\VCal)}(K)'\otimes_K\HH^{n-j}(\PR_K^{n,\rm rig},\ECal^{i}) \rightarrow 0
\end{eqnarray*}
for $j=0,\ldots, n-1.$ For $j=n,$ there is an identification
$$\ECal^i(U^m)^n=\HH^0(\PR_K^{n,\rm rig},\ECal^{i}).$$
These filtrations are compatible with $\G$-equivariant morphisms between the involved sheaves.
\end{thm}
\begin{proof}
The proof is in complete analogy with the one of the corresponding result \cite[Cor.\ 2.2.9]{OrlikEVB}, the notable exception again being avoidance of the use of duals. The compatibility assertion is proved in \cite[Lemma 4]{OrlikDeRham}.
\end{proof}

\subsection{Computation of the Rigid Cohomology Modules}\label{ThreeFourThree}

From Theorem \ref{DescMod}, one obtains a $\G(\VCal)$-filtration $\dlim_{m\in\N}\ECal^i(U^m)^{\bullet}$ of each 
$$\HH^0(\PR_K^{n,\rm rig},\dlim_{m\in\N}{j_{m}}_*j_{m}^*\ECal^i)=\dlim_{m\in\N}\HH^0(\PR_K^{n,\rm rig},{j_{m}}_*j_{m}^*\ECal^i)=\dlim_{m\in\N}\HH^0(U^m,\ECal^i),$$
$i\in\{0,\ldots,n\},$ (which is also compatible with morphisms between the involved sheaves) such that each filtration step appears in a short exact sequence
\begin{eqnarray*}
0&\rightarrow& \dlim_{m\in\N}\Ind_{\Parbf_{(j+1,n-j)}(\VCal)(1)}^{ \G(\VCal)}(\tilde\HH_{]\PR_k^{j}[_{\lambda_m}}^{n-j}(\PR_K^{n,\rm rig},\ECal^{i})\otimes\mathcal{K}_{n-j}) \rightarrow \dlim_{m\in\N}(\ECal^i(U^m)^{j}\slash\ECal^i(U^m)^{j+1})\\
&\rightarrow& v_{\Parbf_{(j+1,1^{n-j})}(\VCal)(1)}^{\G(\VCal)}(K)'\otimes_K\HH^{n-j}(\PR_K^{n,\rm rig},\ECal^{i}) \rightarrow 0
\end{eqnarray*}
for $j=0,\ldots, n-1.$ Here, one of course uses the fact that taking the direct limit $\dlim_{m\in\N}$ in this context preserves exactness and thus is also compatible with taking quotients. For $j=n,$ one gets
$$\dlim_{m\in\N}\ECal^i(U^m)^n=\dlim_{m\in\N}\HH^0(\PR_K^{n,\rm rig},\ECal^{i})=\HH^0(\PR_K^{n,\rm rig},\ECal^{i}).$$
These filtrations can now be used to compute the rigid cohomology modules of $\Drin$ as $ G$-modules. The methods used are those of \cite{OrlikDeRham} by Orlik.\\

First of all, compatibility with $ G$-morphisms (see Theorem \ref{DescMod}) gives complexes 
\begin{eqnarray*}
0&\rightarrow&\dlim_{m\in\N}\ECal^0(U^{m})^{j}\slash\ECal^0(U^{m})^{j+1}\rightarrow\dlim_{m\in\N}\ECal^1(U^{m})^{j}\slash\ECal^1(U^{m})^{j+1}\rightarrow\ldots\\
&\rightarrow&\dlim_{m\in\N}\ECal^{n}(U^{m})^{j}\slash\ECal^{n}(U^{m})^{j+1}\rightarrow 0
\end{eqnarray*}
for $j=0,\ldots,n-1$ (induced by the complex $\ECal^{\bullet}$) the totality of which can be considered as the $E_0$-page of the spectral sequence induced by the filtered complex
$$0\rightarrow\dlim_{m\in\N}\ECal^0(U^{m})\rightarrow\dlim_{m\in\N}\ECal^1(U^{m})\rightarrow\ldots\rightarrow\dlim_{m\in\N}\ECal^r(U^{m})\rightarrow 0$$
computing $\HH_{\rm rig}^{\ast}(\Drin\slash K,\ECal),$
i.e.\
\begin{equation}\label{LastSpecSeq}
E_0^{r,s}=\dlim_{m\in\N}\ECal^{r+s}(U^{m})^r\slash \dlim_{m\in\N}\ECal^{r+s}(U^{m})^{r+1}\Longrightarrow_r\HH_{\rm rig}^{r+s}(\Drin\slash K,\ECal). 
\end{equation}
Depending on some cohomological information about $\ECal,$ one might now be able to compute this spectral sequence and thus $\HH_{\rm rig}^{\ast}(\Drin\slash K,\ECal)$ explicitly. As an illustration, consider again 
$$\ECal=\Om_{\PR_K^{n,\rm rig}},$$ 
i.e.\ 
$$\ECal^i=\Omega_{\PR_K^{n,\rm rig}\slash K}^i,$$ 
which has the following property:
\begin{lemma}\label{KEYLEMMA} The complex 
 \begin{eqnarray*}
	0&\rightarrow& \dlim_{m\in\N}\tilde\HH_{]\PR_k^{j}[_{ 	\lambda_m}}^{n-j}(\PR_K^{n,\rm rig},\Omega_{\PR_K^{n,\rm rig}\slash K}^0)\rightarrow \dlim_{m\in\N}\tilde\HH_{]\PR_k^{j}[_{\lambda_m}}^{n-j}(\PR_K^{n,\rm rig},\Omega_{\PR_K^{n,\rm rig}\slash K}^1)\rightarrow\ldots\\
	&\rightarrow&\dlim_{m\in\N}\tilde\HH_{]\PR_k^{j}[_{\lambda_m}}^{n-j}(\PR_K^{n,\rm rig},\Omega_{\PR_K^{n,\rm rig}\slash K}^n)\rightarrow 0
	\end{eqnarray*}
 is acyclic for all $j=0,\ldots,n-1.$ 
\end{lemma}
\begin{proof}(cf.\ the proof of \cite[Prop.\ 5]{OrlikDeRham})
Fix $j\in\{0,\ldots,n-1\}.$ First of all, by construction, there are isomorphisms
$$\tilde\HH_{]\PR_k^{j}[_{\lambda_m}}^{n-j}(\PR_K^{n,\rm rig},\Omega_{\PR_K^{n,\rm rig}\slash K}^i)\cong\coker\left(\HH^{n-j-1}(\PR_K^{n,\rm rig},\Omega_{\PR_K^{n,\rm rig}\slash K}^i)\rightarrow\HH^{n-j-1}(\PR_K^{n,\rm rig}\setminus]\PR_k^{j}[_{\lambda_m},\Omega_{\PR_K^{n,\rm rig}\slash K}^i)\right)$$
for all $i=0,\ldots,n.$ As $n-j-1<n-j,$ it follows from (\ref{TubeDense}) that 
$$\HH_{]\PR_k^{j}[_{\lambda_m}}^{n-j-1}(\PR_K^{n,\rm rig},\Omega_{\PR_K^{n,\rm rig}\slash K}^i)=0$$
for all $i=0,\ldots,n.$ Therefore, for each $i\in\{0,\ldots,n\},$ using the long exact sequence from local cohomology, one obtains a short exact sequence 
\begin{eqnarray*}
0 &\rightarrow& \dlim_{m\in\N}\HH^{n-j-1}(\PR_K^{n,\rm rig},\Omega_{\PR_K^{n,\rm rig}\slash K}^i)=\HH^{n-j-1}(\PR_K^{n,\rm rig},\Omega_{\PR_K^{n,\rm rig}\slash K}^i)\\ 
&\rightarrow& \dlim_{m\in\N}\HH^{n-j-1}(\PR_K^{n,\rm rig}\setminus]\PR_k^{j}[_{\lambda_m},\Omega_{\PR_K^{n,\rm rig}\slash K}^i) \rightarrow \dlim_{m\in\N}\tilde\HH_{]\PR_k^{j}[_{\lambda_m}}^{n-j}(\PR_K^{n,\rm rig},\Omega_{\PR_K^{n,\rm rig}\slash K}^i)\rightarrow 0
\end{eqnarray*}
-- again by exactness of $\dlim_{m\in\N}$ -- which then gives rise to a short exact sequence of complexes
\begin{eqnarray*}
0 &\rightarrow& \left(\HH^{n-j-1}(\PR_K^{n,\rm rig},\Omega_{\PR_K^{n,\rm rig}\slash K}^i)\right)_{i=0,\ldots,n}\rightarrow \left(\dlim_{m\in\N}\HH^{n-j-1}(\PR_K^{n,\rm rig}\setminus]\PR_k^{j}[_{\lambda_m},\Omega_{\PR_K^{n,\rm rig}\slash K}^i)\right)_{i=0,\ldots,n}\\ 
&\rightarrow& \left(\dlim_{m\in\N}\tilde\HH_{]\PR_k^{j}[_{\lambda_m}}^{n-j}(\PR_K^{n,\rm rig},\Omega_{\PR_K^{n,\rm rig}\slash K}^i)\right)_{i=0,\ldots,n}\rightarrow 0.
\end{eqnarray*}
Application of the Snake Lemma then yields a long exact sequence of cohomology objects
\begin{eqnarray*}
\ldots&\rightarrow& h^{l-1}\left(\dlim_{m\in\N}\tilde\HH_{]\PR_k^{j}[_{\lambda_m}}^{n-j}(\PR_K^{n,\rm rig},\Omega_{\PR_K^{n,\rm rig}\slash K}^i)\right)_{i=0,\ldots,n} \rightarrow h^l\left(\HH^{n-j-1}(\PR_K^{n,\rm rig},\Omega_{\PR_K^{n,\rm rig}\slash K}^i)\right)_{i=0,\ldots,n}\\
&\rightarrow& h^l\left(\dlim_{m\in\N} \HH^{n-j-1}(\PR_K^{n,\rm rig}\setminus]\PR_k^{j}[_{\lambda_m},\Omega_{\PR_K^{n,\rm rig}\slash K}^i)\right)_{i=0,\ldots,n}\\
&\rightarrow& h^l\left(\dlim_{m\in\N}\tilde\HH_{]\PR_k^{j}[_{\lambda_m}}^{n-j}(\PR_K^{n,\rm rig},\Omega_{\PR_K^{n,\rm rig}\slash K}^i)\right)_{i=0,\ldots,n} \rightarrow\ldots \text{ }\text{ }.
\end{eqnarray*}
Now, because of the fact that 
\begin{equation}\label{Maja}
 h^{\ast}\left(\HH^{n-j-1}(\PR_K^{n,\rm rig},\Omega_{\PR_K^{n,\rm rig}\slash K}^i)\right)_{i=0,\ldots,n} = K[-(n-j-1)]
\end{equation}
by rigid GAGA, cf.\ \cite[4.10.5]{FrPu}, it is sufficient to show that 
\begin{equation}\label{Lilly}
 h^{\ast}\left(\dlim_{m\in\N}\HH^{n-j-1}(\PR_K^{n,\rm rig}\setminus]\PR_k^{j}[_{\lambda_m},\Omega_{\PR_K^{n,\rm rig}\slash K}^i)\right)_{i=0,\ldots,n} = K[-(n-j-1)] 
\end{equation}
to prove the lemma. This will be done by computing the rigid cohomology $\HH_{\rm rig}^{\ast}((\PR_k^n\setminus\PR_k^j)\slash K)$ from a system of strict open neighborhoods of $]\PR_k^n\setminus\PR_k^j[_P$ in $\PR_K^{n,\rm rig}$ and then comparing with the formula which can be obtained from (\ref{Brucke}):\\

By definition, 
$$\PR_k^n\setminus\PR_k^j=\bigcup_{l=j+1}^nD_+(T_l),$$
with closed complement $\PR_k^j=\bigcap_{l=j+1}^nV_+(T_l)$ in $\PR_k^n.$ For each $m\in\N,$ set
$$W^{m}=\PR_K^{n,\rm rig}\setminus ]\PR_k^j[_{\lambda_m}$$
and denote by
$$f_m:W^m\hookrightarrow\PR_K^{n,\rm rig}$$
the inclusion. By construction (see page \pageref{StrOpNeig}), the system $(W^m)_{m\in\N}$ is a cofinal system of strict open neighborhoods of $]\PR_k^n\setminus\PR_k^j[_P$ in $\PR_K^{n,\rm rig}.$ Thus, by definition,
$$\HH_{\rm rig}^{\ast}((\PR_k^n\setminus\PR_k^j)\slash K)=\Hyp^\ast(\PR_K^{n,\rm rig},\dlim_{m\in\N}{f_m}_*f_m^*\Omega_{\PR_K^{n,\rm rig}\slash K}^{\bullet}).$$
Taking the tube of radius $\lambda$ commutes with taking finite intersections of closed subspaces (cf.\ \cite[2.3.5]{LeStumRigCoh}). Therefore,
\begin{eqnarray*}W^m &=& \PR_K^{n,\rm rig}\setminus \bigcap_{l=j+1}^n]V_+(T_l)[_{\lambda_m}\\
&=& \bigcup_{l=j+1}^n\PR_K^{n,\rm rig}\setminus ]V_+(T_l)[_{\lambda_m}
\end{eqnarray*}
and thus each $f_m$ is a quasi-Stein morphism in the sense of \cite[p.\ 20]{LeStumRigCoh}, since $W^m$ is a finite union of affinoid varieties, cf.\ (the proof of) Lemma \ref{LostLemma}, ii). This particularly implies that the higher direct image ${\rm \bf R}^i {f_{m}}_*$ vanishes for $i>0$ so that the above hypercohomology can be computed as
\begin{eqnarray*}
 \Hyp^\ast(\PR_K^{n,\rm rig},\dlim_{m\in\N}{f_m}_*{f_m}^*\Omega_{\PR_K^{n,\rm rig}\slash K}^{\bullet}) &=& \dlim_{m\in\N}\Hyp^\ast(\PR_K^{n,\rm rig},{f_m}_*f_m^*\Omega_{\PR_K^{n,\rm rig}\slash K}^{\bullet})\\
 &=& \dlim_{m\in\N}\Hyp^\ast(W^m,f_m^*\Omega_{\PR_K^{n,\rm rig}\slash K}^{\bullet})\\
 &=& \dlim_{m\in\N}\Hyp^\ast(W^m,\Omega_{\PR_K^{n,\rm rig}\slash K}^{\bullet}).
\end{eqnarray*}
Therefore, it can be computed as the cohomology of the total complex associated with the double complex
$$\left(\dlim_{m\in\N}\bigoplus_{j+1\leq l_0<\ldots< l_r\leq n} \Gamma\left(\bigcap_{e=0}^{r}\PR_K^{n,\rm rig}\setminus ]V_+(T_{l_e})[_{\lambda_m},\Omega_{\PR_K^{n,\rm rig}\slash K}^{s}\right)\right)^{s=0,\ldots,n}_{r=0,\ldots,n-j}$$
which has, say, as $r$-th row the {\v C}ech complex for the sheaf $\Omega_{\PR_K^{n,\rm rig}}^{r},$ $r=0,\ldots,n,$ associated with the above covering of $W^{\lambda_m}.$ Taking cohomology along the rows of this double complex yields the $E_1$-page of a spectral sequence
$$E_1^{r,s}=\dlim_{m\in\N}\HH^s(\PR_K^{n,\rm rig}\setminus]\PR_k^j[_{\lambda_m},\Omega_{\PR_K^{n,\rm rig}\slash K}^{r})\Longrightarrow\HH_{\rm rig}^{\ast}((\PR_k^n\setminus\PR_k^j)\slash K).$$
Combining (\ref{TubeDense}) with the long exact local cohomology sequence and (\ref{Maja}), one now obtains the desired result (\ref{Lilly}) from computing the $E_2$-page of this spectral sequence: it can be seen from (\ref{Brucke}), again using the long exact cohomology sequence for rigid cohomology, that 
$$\HH_{\rm rig}^\ast((\PR_k^n\setminus\PR_k^j)\slash K)=\bigoplus_{l=0}^{n-j-1}K[-2l].$$ 
\end{proof}

It follows that the complex
\begin{eqnarray*}0 &\rightarrow& \dlim_{m\in\N}\Ind_{\Parbf_{(j+1,n-j)}(\VCal)(1)}^{\G(\VCal)}\left(\tilde\HH_{]\PR_k^{j}[_{\lambda_m}}^{n-j}(\PR_k^{n,\rm rig},\Omega_{\PR_K^{n,\rm rig}\slash K}^0)\otimes\mathcal{K}_{n-j} \right)\rightarrow\ldots\\
&\rightarrow& \dlim_{m\in\N}\Ind_{\Parbf_{(j+1,n-j)}(\VCal)(1)}^{ \G(\VCal)}\left(\tilde\HH_{]\PR_k^{j}[_{\lambda_m}}^{n-j}(\PR_k^{n,\rm rig},\Omega_{\PR_K^{n,\rm rig}\slash K}^n)\otimes\mathcal{K}_{n-j}\right)\rightarrow 0
\end{eqnarray*}
is acyclic for all $j=0,\ldots,n-1,$ hence the $E_1$-page of the spectral sequence (\ref{LastSpecSeq}) computes as
\begin{eqnarray*}
E_1^{r,s} &=& h^{s}(E_0^{r,\bullet})\\
&=& h^{s}(\dlim_{m\in\N}\FCal^{r+\bullet}(U^{m})^r\slash\FCal^{r+\bullet}(U^{m})^{r+1})\\
&=& h^{s}(v_{\Parbf_{(r+1,1^{n-r})}(\VCal)(1)}^{\G(\VCal)}(K)'\otimes_K\HH^{n-r}(\PR_k^{n,\rm rig},\Omega_{\PR_K^{n,\rm rig}\slash K}^{r+\bullet}))\\
&=& \begin{cases}
     v_{\Parbf_{(r+1,1^{n-r})}(\VCal)(1)}^{\G(\VCal)}(K)' &\text{\rm if } r\in\{0,\ldots,n\}, s=n-2r\\
     0 &\text{else.}
    \end{cases}
\end{eqnarray*}
From this shape it follows that the spectral sequence degenerates on its $E_1$-page, i.e.\ 
$$E_1^{r,s}=E_\infty^{r,s}=\gr^r(h^{r+s}(\dlim_{m\in\N}\Omega_{\PR_K^{n,\rm rig}\slash K}^\bullet(U^{m})))=\gr^r(\HH_{\rm rig}^{r+s}(\Drin\slash K)).$$
Furthermore, for each $I\subsetneq\Delta,$ the natural identification 
$$\G(\VCal)\slash\Parbf_I(\VCal)(1)\rightarrow G\slash P_I$$
of sets with $\G(\VCal)$-action yields an isomorphism 
$$v_{\Parbf_I(\VCal)(1)}^{\G(\VCal)}(K)\xrightarrow{\sim}v_{P_I}^{G}(K)$$
of $\G(\VCal)$-modules. Therefore, the following theorem is (re)proved:
\begin{thm}\label{RigDrinSecComp} The rigid cohomology of $\Drin$ is given by
 $$\HH_{\rm rig}^\ast(\Drin\slash K)=\bigoplus_{i=0}^nv_{P_{(1+n-i,1^i)}}^{G}(K)'[-i].$$
\end{thm}
\noindent Adding Tate twists then yields the formula that would be obtained from Theorem \ref{RigDrinFirstComp} by using Poincar\'e duality.


\begin{thebibliography}{[11]}
\bibitem{BerGeoRig}{{\scshape P.\ Berthelot}. {\it G\'eom\'etrie rigide et cohomologie des vari\'et\'es alg\'ebriques de caract\'eristique $p$}. Introductions aux cohomologies $p$-adiques (Luminy, $1982$). M\'em.\ Soc.\ Math.\ France (N.S.) No.\ $23$ $(1986)$, $3$, $7$--$32$.}
\bibitem{BGR}{{\scshape S.\ Bosch, U.\ Güntzer, R.\ Remmert}. {\it Non-Archimedean analysis}. A systematic approach to rigid analytic geometry. Grundlehren der Mathematischen Wissenschaften, $261$. Springer-Verlag, Berlin, $1984$.}
\bibitem{DOR}{{\sc J.-F.\ Dat, S.\ Orlik, M.\ Rapoport}. {\it Period domains over finite and $p$-adic fields}. Cambridge Tracts in Mathematics, $183$. Cambridge University Press, Cambridge, $2010$.}
\bibitem{DeLu}{{\sc P.\ Deligne, G.\ Lusztig}. {\it Representations of reductive groups over finite fields}. Ann.\ of Math.\ $(2)$ $103$ $(1976)$, no.\ 1, $103$--$161$.}
\bibitem{DrinfeldEM}{{\sc V.G.\ Drinfeld}. {\it Elliptic modules}. (Russian) Mat.\ Sb.\ (N.S.) $94(136)$ $(1974)$, $594$--$627$.} 
\bibitem{DrinfeldCovs}{{\sc V.G.\ Drinfeld}. {\it Coverings of $p$-adic symmetric domains}. (Russian) Funkcional.\ Anal.\ i Prilo{\v z}en.\ $10$ $(1976)$, no.\ $2$, $29$--$40$.}
\bibitem{EsVie}{{\sc H.\ Esnault, E.\ Viehweg.} {\it Lectures on Vanishing Theorems}. DMV Seminar, $20$. Birkh\"auser Verlag, Basel, $1992.$}
\bibitem{FrPu}{{\sc J.\ Fresnel, M.\ van der Put}. {\it Rigid analytic geometry and its applications}. Progress in Mathematics, $218$. Birkh\"auser Boston, Inc., Boston, MA, $2004$.}
\bibitem{GrKl}{{\sc E.\ Gro{\ss}e-Kl\"onne}. {\it On the crystalline cohomology of Deligne-Lusztig varieties}. Finite Fields Appl.\ $13$ $(2007)$, no.\ $4$, $896$--$921$.}
\bibitem{HuberMZ}{{\sc R.\ Huber}. {\it A generalization of formal schemes and rigid analytic varieties.} Math.\ Z.\ $217$ $(1994)$, no.\ 4, $513$--$551$.}
\bibitem{HuberBook}{{\sc R.\ Huber}. {\it \'Etale cohomology of rigid analytic varieties and adic spaces.} Aspects of Mathematics, E$30$. Friedr.\ Vieweg \& Sohn, Braunschweig, $1996$.}
\bibitem{HumSt}{{\scshape J.E.\ Humphreys}. {\it The Steinberg representation}. Bull.\ Amer.\ Math.\ Soc.\ (N.S.) $16$ $(1987)$, no.\ $2$, $247$--$263$.}
\bibitem{Kiehl}{{\scshape R.\ Kiehl}. {\it Theorem A und Theorem B in der nichtarchimedischen Funktionentheorie}. Invent.\ math.\ $2$ $(1967)$, $256$--$273$.}
\bibitem{LeStumRigCoh}{{\sc B.\ Le Stum}. {\it Rigid Cohomology}. Cambridge Tracts in Mathematics, $172$. Cambridge University Press, Cambridge, $2007$.}
\bibitem{OrlikDissCol}{{\scshape S.\ Orlik}. {\it Kohomologie von Periodenbereichen \"uber endlichen K\"orpern}. Dissertation. Universit\"at zu K\"oln. $1999$.}
\bibitem{OrlikDiss}{{\scshape S.\ Orlik}. {\it Kohomologie von Periodenbereichen \"uber endlichen K\"orpern}. J.\ Reine Angew.\ Math.\ $528$ $(2000)$, $201$--$233$.}
\bibitem{OrlikEVB}{{\scshape S.\ Orlik}. {\it Equivariant vector bundles on Drinfeld's upper half space}. Invent.\ math.\ $172$ $(2008)$, no.\ $3$, $585$--$656$.}
\bibitem{OrlikDeRham}{{\scshape S.\ Orlik}. {\it The de Rham cohomology of Drinfeld's upper half space}. M\"unster J.\ Math.\ $8$ $(2015),$ $169$--$179$.}
\bibitem{OrlikDeLu}{{\scshape S.\ Orlik}. {\it The cohomology of Deligne-Lusztig varieties for the general linear group}. Preprint, arXiv:1302.0482.}
\bibitem{Quillen}{{\sc D.\ Quillen}. {\it Homotopy properties of the poset of nontrivial $p$-subgroups of a group}. Adv.\ in Math.\ $28$ $(1978)$, no.\ $2$, $101$--$128$.}
\bibitem{RapoportICM}{{\sc M.\ Rapoport}. {\it Non-Archimedean period domains}. Proceedings of the International Congress of Mathematicians, Vol.\ $1$, $2$ (Z\"urich, $1994$), $423$--$434$, Birkh\"auser, Basel, $1995$.}
\bibitem{RapoportSantaCruz}{{\sc M.\ Rapoport}. {\it Period domains over finite and local fields}. Algebraic geometry--Santa Cruz $1995$, $361$--$381$, Proc.\ Sympos.\ Pure Math., $62$, Part $1$, Amer.\ Math.\ Soc., Providence, RI, $1997$.}
\bibitem{Rotman}{{\scshape J.J.\ Rotman}. {\it An introduction to homological algebra}. Second edition. Universitext. Springer, New York, $2009$.}
\bibitem{SchSt}{{\scshape P.\ Schneider, U.\ Stuhler}. {\it The cohomology of $p$-adic symmetric spaces}. Invent.\ math.\ $105$ $(1991)$, no.\ 1, $47$--$122$.}
\bibitem{vdPut}{{\sc M.\ van der Put}. {\it Serre duality for rigid analytic spaces}. Indag.\ Mathem.\ (N.S.) \ $3$ ($1992$), no.\ 2, $219$--$235$.}
\end{thebibliography}
\end{document}